\documentclass[reqno,11pt]{amsart}

\usepackage{amsmath,amssymb,amsthm,mathtools}
\usepackage{enumitem}
\usepackage{hyperref}
\usepackage{mathrsfs}
\usepackage{microtype}
\usepackage{xcolor}
\usepackage{marginnote}
 \usepackage{todonotes}

\hypersetup{
    colorlinks=true,
    linkcolor=blue,
    citecolor=blue,
    urlcolor=blue
}

\newtheorem{theorem}{Theorem}[section]
\newtheorem{proposition}[theorem]{Proposition}
\newtheorem{lemma}[theorem]{Lemma}
\newtheorem{corollary}[theorem]{Corollary}
\newtheorem{conjecture}[theorem]{Conjecture}
\newtheorem{question}[theorem]{Question}
\theoremstyle{definition}
\newtheorem{definition}[theorem]{Definition}
\theoremstyle{remark}
\newtheorem{remark}[theorem]{Remark}

\DeclareMathOperator{\tr}{tr}

\DeclareMathOperator{\vol}{vol}

\DeclareMathOperator{\intt}{int}

\DeclareMathOperator{\Per}{Per}

\newcommand{\cQ}{\mathcal Q}
\newcommand{\cE}{\mathcal E}

\DeclareMathOperator{\GL}{GL}
\DeclareMathOperator{\SL}{SL}
\newcommand{\R}{\mathbb{R}}

\newcommand{\E}{\mathbb E}
\DeclareMathOperator{\OO}{O}
\DeclareMathOperator{\SO}{SO}
\newcommand{\N}{\mathbb N}

\newcommand{\1}{\mathbf 1}
\newcommand{\eps}{\varepsilon}
\newcommand{\FK}{\mathrm{FK}}
\newcommand{\HS}{\mathrm{HS}}
\newcommand{\OU}{\mathrm{OU}}
\newcommand{\dd}{\,d}

\newcommand{\rcheb}{r_{\rm Ch}}

\newcommand{\norm}[1]{\left\lVert #1\right\rVert}
\newcommand{\ip}[2]{\left\langle #1,#2\right\rangle}

\usepackage{xcolor}

\title[The Faber--Krahn position and Gaussian inequalities]{The Faber--Krahn position of convex bodies and Gaussian measure inequalities}
\author{Dmitry Faifman}
\author{Iosif Polterovich}

\address{D\'epartement de Math\'ematiques et de Statistique, Universit\'e de
	Montr\'eal, CP 6128 succ Centre-Ville, Montr\'eal, QC H3C 3J7, Canada}
\email{dmitry.faifman@umontreal.ca, iosif.polterovich@umontreal.ca}

\begin{document}

\begin{abstract}
We say that a convex body is in Faber–Krahn position if it minimizes the first Dirichlet eigenvalue within its volume-preserving linear orbit. We prove that this position is unique up to orthogonal transformations, answering a question
of Schmuckenschl\"ager from 2011.  This is a corollary of a new log-convexity property of the first eigenvalue under positive definite linear deformations.
While the centrally symmetric case follows from the Gaussian B-theorem, the extension to arbitrary convex bodies requires a quantitative analysis of conditioned Brownian motion. As consequences, we obtain  several applications to the minimization of the first Dirichlet eigenvalue within linear orbits, including a new proof of the Pólya--Szegő theorem for triangles and its analogue for simplices. We also prove related convexity results for  the first eigenvalue of the Ornstein--Uhlenbeck operator,
the inverse inradius and the planar Cheeger constant. In a  different direction, we show using similar ideas that the Gaussian conjugate Rogers--Shephard inequality due to Milman--Nakamura--Tsuji yields improved Schmuckenschl\"ager-type bounds for intersections and Minkowski sums of centrally symmetric convex bodies.
\end{abstract}

\maketitle

\tableofcontents

\section{Introduction}

\subsection{Motivation and background}

A convex body is a compact convex subset of $\R^n$ with nonempty interior.  We write $|K|$ for its $n$-dimensional volume.  The first Dirichlet eigenvalue, denoted $\lambda_1(K)$, is the first eigenvalue of $-\Delta$ on $\intt (K)$.

A recurring theme at the intersection of convex and spectral geometry is that inequalities for the Gaussian measure of convex sets yield inequalities for the first Dirichlet eigenvalue.  The bridge is probabilistic and can be explained as follows. Let $X_t$ be the Brownian motion in $\R^n$ with generator $\Delta/2$ starting at $x$. We denote by $\mathbb P_x$ the corresponding probability. Let
\[
        \tau_K=\inf\{t>0:X_t\notin K\}
\]
be the first exit time from $K$.  The well-known formula of Kac reads
\begin{equation}\label{eq:kac-intro}
        \frac12 \lambda_1(K)=-\lim_{T\to\infty}\frac1T\log \mathbb P_x(\tau_K>T),
\end{equation}
whenever $x\in\mathrm{int}(K)$.

If we consider a finite sample of the Brownian path, namely
\[
        X^m:=(X_{T/m},X_{2T/m},\ldots,X_T),
\]
then $X^m\in(\R^n)^m$ is a random vector with a certain centered Gaussian probability distribution $\widetilde\gamma_{T, m}$. One may approximate the event $\{\tau_K>T\}=\{X_t\in K ,\quad \forall t\leq T\}$ by the event $\{X^m\in K^m\}$, which has probability $\widetilde \gamma_{T,m}(K^m)$.

Therefore, inequalities for Gaussian measures of convex sets can lead in the limit $m\to \infty$ to inequalities for Brownian survival probabilities, and after taking the limit $T\to\infty$, to inequalities for $\lambda_1$.

Several classical results fit this scheme. The log-concavity of the Gaussian measure readily implies the Brunn--Minkowski inequality for the first eigenvalue
\begin{equation} \label{eq:BM}\lambda_1(tK+(1-t)L)\leq t\lambda_1(K)+(1-t)\lambda_1(L),\end{equation}
which by homogeneity considerations  yields the formally stronger concavity statement \begin{equation} \label{eq:BM2} \lambda_1^{-1/2}(tK+(1-t)L)\geq t  \lambda_1^{-1/2}(K)+(1-t) \lambda_1^{-1/2}(L).\end{equation}  First established by Brascamp--Lieb \cite{BrascampLieb} who studied the diffusion equation stressing PDE methods, the inequality was later approached explicitly by Brownian motion in \cite{Borell_diffusion}. A different proof for convex $K, L$ using tools from convex analysis appeared in \cite{Colesanti}.

Brascamp and Lieb have also shown that the first Dirichlet eigenfunction of a convex domain is log-concave \cite{BrascampLieb}, see also \cite{Borell_hitting}. This is similarly the limit of the log-concavity of $x\mapsto \mathbb P_x(\tau_K>T)$ for any finite $T$.

The celebrated Faber--Krahn
inequality asserts that a Euclidean ball minimizes $\lambda_1(\Omega)$ among domains $\Omega$ of fixed volume, and can be proved using symmetrization. 
Luttinger's generalized isoperimetric inequalities \cite{Luttinger1}, together with the Brascamp--Lieb--Luttinger rearrangement inequality \cite{BrascampLiebLuttinger} (see also \cite{Banuelos_latala_mendez}), imply the stronger fact that, for every fixed $T>0$, the survival probability $\mathbb P_0(\tau_\Omega>T)$ is maximized by the Euclidean ball centered at the origin among sets $\Omega$ of fixed volume containing the origin. Luttinger similarly proved that, for every fixed $T>0$, the Dirichlet heat trace is increased by symmetric decreasing rearrangement, and hence is maximized by the ball.

Schmuckenschl\"ager's inequality \cite{Schmuckenschlager} for centrally symmetric convex bodies states that
\begin{equation}\label{eq:schmuck_intersection}\lambda_1(K\cap L)\leq \lambda_1(K)+\lambda_1(L).\end{equation} 
Its proof is
rather constructive up to an averaging argument of Lieb. However,  it can also be seen as an immediate consequence of the Gaussian correlation inequality \cite{Royen, Latala_correlation} for centrally symmetric  convex bodies in Gaussian space $(\R^n, \gamma)$
\[\gamma(K\cap L)\geq \gamma(K)\gamma(L).\]
Indeed it readily implies, using sampled Brownian motion, that for finite $T$
\[\mathbb P_0(\tau_{K\cap L}>T)\geq \mathbb P_0(\tau_{K}>T) \mathbb P_0(\tau_{L}>T).\]

 \subsection{Main results}
 The B-conjecture was proposed by Banaszczyk, and proved in a stronger form by Brascamp--Lieb \cite{BrascampLieb1975} (see also \cite{Cordero_Eskenazis}) and Cordero-Erausquin, Fradelizi and Maurey \cite{CEFM}, which is why we will henceforth call it the B-theorem. It asserts that if $K\subset(\R^n, \gamma_n)$ is a centrally symmetric convex body in standard Gaussian space, and $S$ is any symmetric matrix, then
\[
        t\longmapsto \gamma_n(e^{tS}K)
\]
is log-concave. Curiously, the elegant proof of Brascamp--Lieb went unnoticed, and only resurfaced more than two decades after the publication of \cite{CEFM}. 

Our first main result is a similar convexity property of $\lambda_1$. While the B-conjecture fails for arbitrary convex bodies, the corresponding spectral statement holds without symmetry assumption. 
\begin{theorem}\label{thm:main-logconvex-intro}

Let $K\subset\R^n$ be a convex body and let $P$ be positive definite.  Then the function
\[
f(t)=\lambda_1(P^tK)
\]
is log-convex on $\R$. That is,
\[
f((1-s)t_0+st_1)\le f(t_0)^{1-s}f(t_1)^s
\]
for all $t_0,t_1\in\R$ and $s\in[0,1]$.

Moreover, if $P\in \SL(n)$ and $P\neq I_n$ then $f(t)$ is strictly log-convex.
\end{theorem}

We first deduce the centrally symmetric case directly from the B-theorem, and then prove the general case by combining Brownian discretization with a quantitative version of the $L^2$ method of \cite{CEFM}, see section \ref{sec:logconv} for more details.

The Brunn--Minkowski inequality \eqref{eq:BM} is a different convexity property of $\lambda_1$.  Theorem \ref{thm:main-logconvex-intro} implies a strengthening of the Brunn--Minkowski inequality in the case $L=PK$, see section \ref{sec:comparison_BM}.

Theorem \ref{thm:main-logconvex-intro} leads to a new affine position of convex bodies.

\begin{definition}
	For a convex body $K\subset\R^n$ , we say that $K$ is in \emph{Faber--Krahn position} if
	\[
	\lambda_1(K)=\inf_{A\in \SL(n)}\lambda_1(AK).
	\]
\end{definition}

We prove that every convex body has a Faber--Krahn position which is unique up to orthogonal transformations. This answers a question asked by Schmuckenschl\"ager \cite{Schmuckenschlager}.  

Recall that a convex body $K$ is in John's position if the maximal volume ellipsoid inside $K$ is a Euclidean ball. It was proved by John that there exists $A\in\SL(n)$, unique up to orthogonal transformation, such that $AK$ is in John's position.
The Faber--Krahn position can be seen as a spectral version of the John position. Indeed, it is well-known that $\lambda_1(K)^{\frac12}$ is comparable, up to dimension-dependent constant, to $1/r(K)$, where $r(K)$ is the inradius.
It holds that $r(gK)$ is maximized, among $g\in\SL(n)$, exactly at the John position, namely when the unique volume maximizing ellipsoid inside $gK$ is a Euclidean ball.
In fact, $r(P^tK)^{-1}$ is log-convex for any positive-definite $P$, similarly to $\lambda_1$ in Theorem \ref{thm:main-logconvex-intro}. This fact is likely well-known, but for lack of reference we include a proof, see Proposition \ref{prop:inradius-logconvex}.

The Euler-Lagrange condition for the Faber--Krahn position has a simple form recorded in \cite{Schmuckenschlager},
see also Proposition \ref{prop:EL}.   If $\psi$ is the positive first Dirichlet eigenfunction, normalized by $\int_K\psi^2=1$, then in Faber--Krahn position
\begin{equation}\label{eq:spectral-isotropy-intro}
        \int_K \nabla \psi\otimes \nabla \psi\,dx=\frac{\lambda_1(K)}{n}I,
\end{equation}
and log-convexity turns it into a necessary and sufficient condition. Thus the ground-state energy is distributed isotropically among all directions in Faber--Krahn position. 

Theorem \ref{thm:main-logconvex-intro} and the uniqueness of the Faber--Krahn position have several immediate consequences.  A convex body whose symmetry group acts irreducibly on $\R^n$ is automatically in Faber--Krahn position.  In particular, the regular simplex minimizes $\lambda_1$ among all simplices of fixed volume, in every dimension. This gives a new proof of 
this fact, which is normally proved using  the Steiner symmetrization. The same argument shows that a regular polygon minimizes $\lambda_1$ among all of its area-normalized linear images.  This gives a novel application: the affine-orbit part of the P\'olya--Szeg\H{o} polygonal conjecture, stating that the regular $m$-gon minimizes the first Dirichlet eigenvalue among all $m$-gons of the same area (see \cite[Conjecture 5.1.19]{LMP} and \cite{BucurBogosel} for background and recent developments). We refer to section \ref{sec:symm} for details.

In addition to Theorem  \ref{thm:main-logconvex-intro} for the inverse inradius, we prove it as well for the planar Cheeger constant. The latter uses the planar characterization of the Cheeger set of a convex body due to Kawohl and Lachand-Robert \cite{KawohlLachandRobert}.
In light of these results, it appears natural to ask whether the first Dirichlet $p$-Laplacian eigenvalue $\lambda_{1,p}(K)$ satisfies the same log-convexity property for general $1< p<\infty$, see section \ref{sec:inradius}.

Our second Gaussian departure point is the Gaussian conjugate Rogers-Shephard inequality of Milman, Nakamura and Tsuji \cite{Milman_Nakamura_Tsuji}, which states that for centrally symmetric convex sets $A,B$ and any centered Gaussian measure $\gamma$,
\begin{equation}\label{eq:MNT-intro}
	\gamma(A)\gamma(B)\le \gamma(A\cap B)\gamma(A+B),
\end{equation}
refining the Gaussian correlation inequality.

The spectral consequence of \eqref{eq:MNT-intro} strengthens Schmuckenschl\"ager's inequality \eqref{eq:schmuck_intersection}. 
\begin{theorem}\label{thm:improved-schmuck}
	Let $K,L\subset\R^n$ be centrally symmetric convex bodies.  Then
	\begin{equation}\label{eq:improved-schmuck}
		\lambda_1(K\cap L)+\lambda_1(K+L)
		\le
		\lambda_1(K)+\lambda_1(L).
	\end{equation}
\end{theorem}

Eq. \eqref{eq:improved-schmuck} similarly holds for the first eigenvalue  of the Ornstein--Uhlenbeck operator with Dirichlet condition.

Finally for the general $p$-Laplacian we formulate an improved Schmuckenschl\"ager inequality as an open problem, and
prove the weaker inequality
\[
        \lambda_{1,p}(K\cap L)^{1/p}
        \le
        \lambda_{1,p}(K)^{1/p}+\lambda_{1,p}(L)^{1/p},
\]
see Theorem \ref{thm:p-schmuck}.

\subsection*{Acknowledgements}  The authors would like to thank Emanuel Milman for explaining to us Theorem \ref{thm:KM}, Liran Rotem  and Rapha\"el Pothier for useful discussions, as well as Richard Laugesen and Rodrigo Ba\~nuelos for helpful comments on an earlier version of the paper.
The research of D.F. was partially supported by an NSERC Discovery Grant. The research of I.P. was partially supported by NSERC and FRQNT.
\subsection*{AI usage disclosure}The authors made use of ChatGPT for mathematical discussions and editorial assistance.  In particular, it has brought to our attention the idea to prove Theorem \ref{thm:improved-schmuck} using inequality \eqref{eq:MNT-intro}. It has also contributed to the development of  arguments in section \ref{subsec:discrete-quasi} and to the proof of Proposition~\ref{prop:cheeger-logconvex}. 
All mathematical claims and proofs were independently checked, revised, and written in their final form by the authors.

\section{Log-convexity, the Faber--Krahn position and applications}

\subsection{The log-convexity theorem}
\label{sec:logconv}
	Consider the space of left cosets $\mathcal P_n= \OO(n) \backslash \GL(n)=\{ \OO(n)g: g\in \GL(n)\}$, which one may identify with the cone of positive definite matrices through $\OO(n)g\mapsto g^Tg$, equipped with the standard Riemannian metric turning it into a symmetric space of rank $n$. Its geodesics are given by $\gamma_{g, S}(t)=\OO(n)e^{tS}g$ with $S$ symmetric and $g\in\GL(n)$, see \cite{Helgason} and \cite[Chapter 6]{Bhatia} for details.

 Thus Theorem \ref{thm:main-logconvex-intro} can be equivalently stated as follows.

\begin{theorem}\label{thm:log_convex2}
	Given a fixed convex body $K\subset\R^n$, define a smooth function on $\mathcal P_n$ by $\Lambda_K(\OO(n)g):=\lambda_1(g K)$. Then $\log\Lambda_K$ is geodesically convex on $\mathcal P_n$, and strictly convex along any geodesic $\gamma_{g, S}$ with $S$ non-scalar.
\end{theorem}

A similar phenomenon occurs for many natural functions of convex bodies, see section \ref{sec:isotropic_discussion} for some examples.

The proof of Theorem \ref{thm:main-logconvex-intro} is given in Section \ref{sec:proof_convexity}. The main difficulty in the argument 
is the passage from centrally symmetric bodies to general convex 
bodies. In the symmetric case the result follows directly from the Gaussian B-theorem, since the relevant conditioned Gaussian measures have vanishing barycenter. Without symmetry, the $L^2$-argument underlying the B-theorem produces an additional barycenter term, and there is no reason for it to vanish. Our key observation is that, after discretizing a Brownian path conditioned to remain in $P^tK$ for a long time, its conditional mean is nearly constant away from the initial and terminal time segments. The inverse covariance matrix of the sampled Brownian path is a discrete Laplacian in the time variable, and therefore it is essentially insensitive to this almost constant component. As a result, the obstruction to log-concavity is sublinear in the length of the path and disappears after division by time in Kac's formula. Establishing this rigorously requires the uniform quasi-stationary estimate, the discretized Kac formula, and the quantitative $L^2$-bound developed in sections \ref{subsec:discrete-quasi}--\ref{sec:proof-logconvex}.

\subsection{Comparison to the Brunn--Minkowski inequality}\label{sec:comparison_BM}
Consider a pair of convex bodies $K$ and $L=PK$ with $P$ positive definite. In order to compare the Brunn--Minkowski inequality to the log-convexity of $\lambda_1(P^tK)$, we will need a simple lemma.

\begin{lemma}\label{lem:positive_monotone}
	If $P_1\leq P_2$ are two commuting positive-definite matrices, then it holds for any convex $K$ that $\lambda_1(P_1K)\geq \lambda_1(P_2K)$.  
\end{lemma}
\begin{proof}
	It holds that
\begin{align*}\lambda_1(P_jK)=\inf_{u\in H^1_0(P_jK)}\frac{\int_{P_jK} |\nabla u|^2}{\int_{P_jK} u^2}&= \inf_{v\in H^1_0(K)}\frac{\int_{K} |P_j^{-1}\nabla v|^2}{\int_{K} v^2}\\&=  \inf_{v\in H^1_0(K)}\frac{\int_{K} \langle P_j^{-2}\nabla v,\nabla v\rangle}{\int_{K} v^2}.\end{align*}
As $P_1^{-2}\geq P_2^{-2}$, this concludes the proof.
\end{proof}
As $ \frac12 (I_n+P)K\subset\frac12 (K+PK)$, one deduces from domain monotonicity and Lemma \ref{lem:positive_monotone} that
\begin{equation}\label{eq:3_eigenvalues}\lambda_1\left( \frac12 (K+PK)\right) \leq  \lambda_1\left( \frac12 (I_n+P)K\right)  \leq  \lambda_1( P^{1/2}K).\end{equation}
	
Denote the $p$-mean by
\[
M_p(a,b)
=
\left(\frac12a^p+\frac12 b^p\right)^{1/p}.\]
Recall that $M_0(a,b)=\sqrt{ab}$, and $M_p(a,b)\leq M_q(a,b)$ for $p\leq q$. 
	
Now the Brunn--Minkowski inequality \eqref{eq:BM} implies 
\[ \lambda_1\left( \frac12 (K+PK)\right) \leq M_{-1/2}(\lambda_1(K), \lambda_1(PK)),\]
while Theorem \ref{thm:main-logconvex-intro} implies

\[\lambda_1( P^{1/2}K)\leq M_0(\lambda_1(K), \lambda_1(PK)).\]

The next corollary shows that the middle eigenvalue in \eqref{eq:3_eigenvalues} satisfies

\[\lambda_1\left( \frac12 (I_n+P)K\right) \leq M_{-1/2}(\lambda_1(K), \lambda_1(PK)),\] thus refining the Brunn--Minkowski inequality in this case.
\begin{corollary}\label{cor:BM_improved}
	Let $K\subset\R^n$ be a convex body, and $P$ a positive definite matrix. Then 
	\[t\mapsto \lambda_1\Big(\big((1-t)I_n+tP\big)K\Big)^{-1/2}\]
	is concave.
\end{corollary}
\begin{proof}
Let \(0< t< 1\). For any $0<s<1$, the weighted AM--GM inequality gives
\[(1-t)I_n+t P=(1-s)\frac{1-t}{1-s}I_n+s\frac{t}{s}P\geq c_{t,s}P^s,\]
where
\[c_{t,s}=\frac{(1-t)^{1-s}t^s}{(1-s)^{1-s}s^s}.\]

Putting
\[a=\lambda_1(K),\qquad b=\lambda_1(PK),\]
 Lemma \ref{lem:positive_monotone} together with Theorem \ref{thm:main-logconvex-intro} yield
\[\lambda_1\Big(\big((1-t)I_n+t P\big)K\Big)\leq
c_{t,s}^{-2}\lambda_1(P^sK) \leq c_{t,s}^{-2}a^{1-s}b^s. \]
The right hand side is minimized in $(0, 1)$ by $s$ which satisfies
\[ \frac{s}{1-s}=\frac{t}{1-t}\sqrt{\frac{a}{b}},\]
and so we get 
\[\inf_{0<s<1}c_{t,s}^{-2}a^{1-s}b^s=\left(
\frac{1-t}{\sqrt{a}}+\frac{t}{\sqrt{b}}\right)^{-2},\]
implying
\[ \lambda_1\bigl(((1-t)I_n+t P)K\bigr) ^{-1/2}\geq  (1-t)\lambda_1(K)^{-1/2}+t\lambda_1(PK)^{-1/2}.\]
For $t$ in a general interval $(t_0, t_1)$, simply apply the previous inequality with $(K, P)$ replaced by \[K'=((1-t_0)I_n+t_0P)K, \quad P'=((1-t_0)I_n+t_0P)^{-1}((1-t_1)I_n+t_1P).\] 
\end{proof}
\subsection{The Faber--Krahn position}

 The following result is a key consequence of Theorem \ref{thm:main-logconvex-intro}.
\begin{theorem}[Existence and uniqueness of Faber--Krahn position]\label{thm:FK-position}
Let $K\subset\R^n$ be a convex body. Then there exists $A_0\in \SL(n)$ such that $A_0K$ is in Faber--Krahn position, that is
\[
        \lambda_1(A_0K)=\inf_{A\in\SL(n)}  \lambda_1(AK).
\]
If $A_0K$ and $A_1K$ are both in  Faber--Krahn position, then there is an orthogonal map $U\in \OO(n)$ such that
\[
        A_1=U A_0.
\]
\end{theorem}
\begin{definition}
Define the $\GL(n)$-invariant normalized eigenvalue corresponding to the Faber--Krahn position by
\[
\Lambda_{\FK}(K)=\inf_{A\in \SL(n)}\lambda_1(AK)|K|^{2/n}=\inf_{A\in \GL(n)}\lambda_1(AK)|AK|^{2/n}.
\]
\end{definition}

It is not hard to see that $K\mapsto\Lambda_{\FK}(K)$ is continuous with respect to the Hausdorff metric, see Proposition \ref{prop:FK_continuity}.

The uniqueness of the Faber--Krahn position renders the vanishing of the first variation a necessary and sufficient condition for minimization.

\begin{proposition}[Euler-Lagrange equation]\label{prop:EL}
Let $K$ be in  Faber--Krahn position and let $\psi$ be the positive first Dirichlet eigenfunction on $K$, normalized by
\[
        \int_K \psi^2\,dx=1.
\]
Then
\begin{equation}\label{eq:EL-energy}
        \int_K \nabla \psi\otimes \nabla \psi\,dx
        =\frac{\lambda_1(K)}{n}I.
\end{equation}

Moreover, if eq. \eqref{eq:EL-energy} holds, then $K$ is in Faber--Krahn position.

\end{proposition}
\begin{remark}
	Formula \eqref{eq:EL-energy} is equivalent to the boundary stress identity 
	\begin{equation}\label{eq:boundary-stress}
		\int_{\partial K}(\partial_\nu \psi)^2 x\otimes \nu\,d\sigma
		=\frac{2\lambda_1(K)}{n}I,
	\end{equation}
	obtained through the application of Hadamard's variational formula \cite[Theorem 4.1]{Frank}.
   Here  $\partial_\nu \psi$ is  understood in the Sobolev trace sense.
\end{remark}

One can recast the uniqueness of the Faber--Krahn position as the possibility to attach a Euclidean metric to a convex body in a $\GL$-equivariant manner, a spectral analogue to the John ellipsoid.  Denote by $\lambda_1(K, E)$ the first eigenvalue of $K$ for the negative Dirichlet Laplacian, defined by the Euclidean structure for which the centered ellipsoid $E$ is the unit ball.

\begin{proposition}
	Among centered ellipsoids of fixed volume, there is a unique ellipsoid $E=E_0$ minimizing $\lambda_1(K, E)$. 
\end{proposition}
\begin{proof} It holds for $g\in\GL(n)$, $\lambda_1(gK, E)=\lambda_1(K, g^{-1}E)$, and so the uniqueness of an ellipsoid $E=E_0$ of fixed volume minimizing $\lambda_1(K, E)$ is immediate from Theorem \ref{thm:FK-position}. %Equivariance is equally clear.
\end{proof}

\begin{definition}
The \emph{Faber--Krahn ellipsoid} of $K$ %, denoted $E_{\FK}(K)$, 
is defined as $E_{\FK}=x_\psi+\alpha E_0$, where $\alpha>0$ is such that  $\lambda_1(K, \alpha E_0)=\lambda_1(B^n)$, and $x_\psi$ is the  unique maximum of the first eigenfunction $\psi$ of $K$ with respect to $E_0$.
\end{definition}
\noindent The correspondence $K\mapsto E_{\FK}(K)$ is then translation- and $\GL(n)$-equivariant. $K$ is in Faber--Krahn position if and only if $E_{\FK}(K)$ is a Euclidean ball.
 
\begin{lemma}
	The Faber--Krahn ellipsoid is well-defined.
\end{lemma}
 \begin{proof}
 	The only point to check is the uniqueness of the maximum of the first Dirichlet eigenfunction.
 	By Andrews-Clutterbuck \cite[eq. (6)]{Andrews_Clutterbuck},
 	\[ \langle \nabla \log \psi(y)-\nabla \log \psi(x), \frac{y-x}{|y-x|}\rangle\leq -\frac{2\pi}{D}\tan\frac{\pi|y-x|}{2D} ,\]
 	where $D=\mathrm{Diam}(K)$. In particular, $\psi$ can't have two distinct critical points,
 and thus it has a unique maximum.
  	\end{proof}
One could fix the center of the Faber--Krahn ellipsoid in a variety of natural manners, for instance by letting it lie at the centroid of $\psi$ or $\psi^2$. However, it seems plausible that with the maximum point definition, we would ultimately obtain a family of ellipsoids corresponding to the $p$-Laplacian, which would continuously interpolate between the John and the Faber--Krahn ellipsoids. 

\subsection{Comparison with other affine positions}\label{sec:isotropic_discussion}

Many natural positions of convex bodies correspond to an extremal value of some geometric quantity $F(K)$ such that $F(P^tK)$ is log convex/concave for positive definite $P$.
In addition to the John position which maximizes the inradius, this is the case for the isotropic position, which minimizes the moment of inertia $\int_K |x|^2dx$, the minimal surface area position, and the minimal mean width position. To our knowledge, the Faber--Krahn position is the first case where the function $F(K)=\lambda_1(K)$ is of spectral rather than of geometric nature.

Moreover, natural positions of convex bodies are often characterized by the isotropicity of some naturally associated measure \cite{Giannopoulos_milman}. This is indeed the case in all aforementioned positions, the role of the isotropic measure taken by the volume measure on $K$ in the isotropic position, the surface area measure $dS_K(\theta)$ on the sphere of unit normals in the minimal surface area position, and $d\nu_K(\theta)=h_K(\theta)d\theta$ on the same sphere in the minimal mean width position. John's position is characterized by the existence of a delta measure supported at the points where the maximal ellipsoid touches the boundary of $K$.

The Faber--Krahn position is no different. Indeed if $K$ is in Faber--Krahn position, and $\psi$ is its first Dirichlet eigenfunction, then by eq. \eqref{eq:EL-energy}, $(\nabla \psi)_*(dx)$ is an isotropic measure.

\subsection{Symmetry and orbit minimizers}
\label{sec:symm}

Let $K\subset\R^n$ have its center of mass at the origin. Let $G(K)\subset \OO(n)$ be the symmetry group of $K$.

\begin{corollary}[Irreducible symmetries]\label{cor:irreducible}
Let $K\subset\R^n$ be a convex body, and let $G(K)$ be its group of orthogonal symmetries.  If $G(K)$ acts irreducibly on $\R^n$, then $K$ is in Faber--Krahn position.  Consequently,
\[
        \lambda_1(AK)\geq \lambda_1(K)
\]
for all $A\in \GL(n)$ with $|\det A|=1$.  Equality holds only when $A$ is orthogonal.
\end{corollary}

\begin{proof}
We give two proofs.

First, let $A\in \SL(n)$ be non-orthogonal.  For every $g\in G(K)$, the linear map
\[
        h_g:=AgA^{-1}
\]
stabilizes $AK$.  If $h_g$ is non-orthogonal for some $g\in G(K)$, then Theorem~\ref{thm:FK-position} shows that $AK$ cannot be in Faber--Krahn position.  Otherwise, $AgA^{-1}$ is orthogonal for every $g\in G(K)$.  Therefore
\[
        g^T A^T A g=A^T A,
        \qquad g\in G(K).
\]
Thus the positive symmetric matrix $A^T A$ commutes with the irreducible action of $G(K)$.  Since $A^T A$ is self-adjoint, Schur's lemma implies that $A^T A$ is a scalar matrix.  Since $A\in \SL(n)$, this scalar is $1$, contradicting that $A$ is non-orthogonal.  Hence no non-orthogonal image $AK$ can be in Faber--Krahn position.  Since a Faber--Krahn position exists, it must hold that $K$ itself is in Faber--Krahn position together with its orthogonal images.

The second proof makes use of the variational characterization of the Faber--Krahn position.  Let $\psi$ be the positive $L^2$-normalized first eigenfunction on $K$, and put
\[
        E=\int_K \nabla \psi\otimes\nabla \psi\,dx.
\]
For each $g\in G(K)$, the simplicity of the first eigenvalue gives $\psi\circ g=\psi$, and hence $gEg^{-1}=E$.  Since $E$ is symmetric and $G(K)$ acts irreducibly, Schur's lemma gives $E=cI$.  Taking traces yields $c=\lambda_1(K)/n$.  Thus $K$ satisfies the Euler-Lagrange condition \eqref{eq:EL-energy}, and Proposition~\ref{prop:EL} implies that $K$ is in Faber--Krahn position.
\end{proof}
We draw some conclusions on the minimality of $\lambda_1(K)$ within its fixed volume linear orbit for various  bodies $K$. Results on extrema of eigenvalues within the linear orbit but with a different normalization  appeared in \cite{Laugesen1, Laugesen2}.

\begin{corollary}[Simplices]\label{cor:simplex}
Among all $n$-simplices of fixed volume in $\R^n$, the regular simplex uniquely minimizes the first Dirichlet eigenvalue.
\end{corollary}

\begin{proof}
Every $n$-simplex is an affine image of a regular simplex.  The symmetry group of the regular simplex acts irreducibly on $\R^n$.  The result follows from Corollary \ref{cor:irreducible}.
\end{proof}

\begin{remark}
	Corollary \ref{cor:simplex} in dimension $2$ recovers the classical P\'olya--Szeg\H{o} theorem for triangles, originally proved with the Steiner symmetrization. In fact, the statement in general dimension can be similarly proved using a carefully chosen sequence of Steiner symmetrizations perpendicular to the edges, as was done by  Hadwiger \cite{Hadwiger} to prove the corresponding isoperimetric inequality for simplices, see also \cite{Boroczky_Kovacs}. We refer also to \cite{Rolling} for another Brownian-motion-based proof, which is nevertheless rooted in Steiner symmetrization.

\end{remark}

Similarly one shows 
\begin{corollary}\label{cor:octahedra}
Among all centrally symmetric polytopes in $\R^n$ of fixed volume with $2n$ vertices, the regular cross-polytope minimizes $\lambda_1$.
\end{corollary}
\begin{proof}
	It is easy to see that all such polytopes lie in the linear orbit of the regular cross-polytope.  
\end{proof}
Once again, this result can be obtained using Steiner symmetrizations, see e.g. \cite[Proposition 9.1]{Gruber}.

One can similarly  deduce that regular polygons in  $\R^2$ and platonic solids in $\R^3$ are in Faber--Krahn position.
Nevertheless, all those cases can also be deduced by carefully chosen Steiner symmetrizations.

\begin{remark}
	Steiner symmetrization does not directly apply to chiral bodies, namely $K$ for which $G(K)\subset \SO(n)$. 
	A well-known example of a chiral convex body whose  symmetry group acts irreducibly is the snub cube, an Archimedean solid with 38 faces consisting of 6 squares and 32 equilateral triangles \cite[Section 21]{Conway_symmetries}.
	Other examples include planar convex bodies with cyclic symmetry groups of order at least $3$, the snub dodecahedron in $\R^3$, and the snub $24$-cell in $\R^4$. Furthermore, chiral polytopes with symmetry group $A_{n+1}$, the even permutations acting irreducibly on $\R^n=\{x\in\R^{n+1}:x_1+\dots+ x_{n+1}=0\}$, exist for any $n\geq 3$ \cite[Theorem C]{Friese_Ladisch}. 
	
	Thus all such bodies are in Faber-Krahn position by Corollary \ref{cor:irreducible}.	
	While in some cases a symmetrization procedure tailored to the symmetry group could potentially replace the Steiner symmetrization to arrive at this conclusion, it is not always clear if or how this can be done.
\end{remark}

\subsection{Inradius and the $p$-Laplacian}
\label{sec:inradius}
The log-convexity phenomenon of Theorem \ref{thm:main-logconvex-intro} is not restricted to $\lambda_1$. Recall that $r(K)$ denotes the inradius of $K$.
\begin{proposition}\label{prop:inradius-logconvex}
Let $K\subset\R^n$ be a convex body and let $P>0$.  Then
\[
        t\longmapsto r(P^tK)^{-1}
\]
is log-convex. Moreover, $\SL(n)\ni A\mapsto r(AK)^{-1}$ has a unique minimum up to orthogonal transformation.
\end{proposition}
The uniqueness of the inradius maximizing position follows directly from John's theorem. Indeed, if $r(AK)$ is maximized by $A\in \SL(n)$ then $AK$ must be in John position. For if $TK$ is in John's position, then the inradius of the maximal volume ellipsoid of $TK$, which is a Euclidean ball, cannot exceed $r(AK)$ by maximality of the latter and so $r(TK)=r(AK)$, and by the uniqueness of the John position, $T=UA$ with $U$ an orthogonal transformation.

More generally, one may consider a $p$-Laplacian analogue.  Let
\[
\lambda_{1,p}(K)=\inf_{v\in W^{1,p}_0(K)\setminus\{0\}}
\frac{\int_K |\nabla v|^p\,dx}{\int_K |v|^p\,dx}.
\]

As $p\to\infty$, one has $\lambda_{1,p}(K)^{1/p}\to r(K)^{-1}$ \cite{infinity_laplacian}. In light of Theorem \ref{thm:main-logconvex-intro} and Proposition \ref{prop:inradius-logconvex}, it is natural to ask if log-convexity holds in general.

\begin{question}\label{qu:lambdap_convex}
	Is the functional $t\mapsto \lambda_{1,p}(P^tK)$ log-convex for all $1<p<\infty$, $p\neq 2$? 
\end{question}

It is not hard to show, as in Proposition \ref{prop:minimum_existence}, that for any $1<p<\infty$ the minimum of $\lambda_{1,p}(AK)$ over all $A\in\SL(n)$ is attained. The uniqueness of such a position is much less clear. 

\begin{question}
	Is the minimum of $\lambda_{1,p}(AK)$, $A\in \SL(n)$ unique up to orthogonal transformation?
\end{question}

Note that log-convexity alone would not suffice, as $\lambda_{1,p}(P^tK)$ is not automatically analytic in $t$ and so strict log-convexity would not be immediate.

Such a conjectural $p$-Faber--Krahn position would interpolate between the spectral Faber--Krahn position and the John position.

Recall the Cheeger constant
\[
h(K)=\inf_{E\subset K}\frac{\Per(E)}{|E|}.
\]
Letting $p\to1$, $\lambda_{1,p}(K)\to h(K)$ \cite{one_laplacian}. We show log-convexity of the Cheeger constant in dimension $2$.

\begin{proposition}\label{prop:cheeger-logconvex}
Let $K\subset\R^2$ be a convex body and let $P>0$.  Then
\[
        t\longmapsto h(P^tK)
\]
is log-convex.
\end{proposition}
This renders the following limiting case of Question \ref{qu:lambdap_convex} particularly plausible.
\begin{conjecture}
	The function $t\mapsto h(P^tK)$ is log-convex in arbitrary dimension.
\end{conjecture}

The proofs of Propositions \ref{prop:inradius-logconvex} and \ref{prop:cheeger-logconvex} are given in Section \ref{subsec:inradius-cheeger-proofs}.

\subsection{Heat trace}

Let
\[
        Z_K(\tau)=\operatorname{Tr}(e^{\tau\Delta_K/2})=\sum_{j\ge1}e^{-\lambda_j(K)\tau/2}=\int_K p^{K, D}_\tau(x, x)dx
\]
be the Dirichlet heat trace of $\Delta/2$, where $p^{K, D}_\tau(x, y)$ is the Dirichlet heat kernel, that is the Schwartz kernel of $e^{\tau\Delta_K/2}$. Luttinger \cite{Luttinger1} gives heat trace versions of Faber--Krahn-type theorems using symmetrization results.  For instance, the classical Steiner symmetrization proof of the P\'olya--Szeg\H{o} theorem for triangles may be strengthened to
\[
        Z_T(\tau)\le Z_{T_{\rm eq}}(\tau),\qquad \tau>0,
\]
for every triangle $T$ with the same area as the equilateral triangle $T_{\rm eq}$.

The B-theorem implies an affine-orbit statement in the centrally symmetric case.

\begin{theorem}\label{thm:heattrace}
Let $K\subset\R^n$ be centrally symmetric and convex, and let $P>0$.  For every $\tau>0$, the function
\[
        t\longmapsto Z_{P^tK}(\tau)
\]
is log-concave.  Consequently, if the symmetry group of $K$ acts irreducibly, then
\[
        Z_{AK}(\tau)\le Z_K(\tau)
\]
for all $A\in \SL(n)$ and all $\tau>0$.
\end{theorem}

\subsection{The Ornstein--Uhlenbeck operator}

Let $d\gamma_n=(2\pi)^{-n/2}e^{-|x|^2/2}dx$.  For a domain $K$ define the first Dirichlet Ornstein--Uhlenbeck eigenvalue by
\[
\lambda^{\OU}_1(K)=
\inf_{u\in H^1_0(K,\gamma_n)\setminus\{0\}}
\frac{\int_K |\nabla u|^2\,d\gamma_n}{\int_K u^2\,d\gamma_n}.
\]
Equivalently, $\lambda^{\OU}_1(K)$ is the first Dirichlet eigenvalue of
\[
-L_\gamma=-\Delta+x\cdot\nabla
\]
in $L^2(K,\gamma_n)$.

\begin{proposition}\label{prop:OU_convex}
	Let $K\subset\R^n$ be a centrally symmetric convex body and let $P>0$. Then
	\[
	t\longmapsto \lambda^{\OU}_1(P^tK)
	\]
	is convex.
\end{proposition}

\section{Log-convexity - proofs}\label{sec:proof_convexity}

\subsection{Preliminaries}
\subsubsection{Brownian discretization}

We use Brownian motion with generator $\Delta/2$.  Thus the transition density is
\begin{equation}\label{eq:transition-density}
        p_t(x,y)=(2\pi t)^{-n/2}\exp\left(-\frac{|x-y|^2}{2t}\right).
\end{equation}
With this convention, the Dirichlet semigroup is $e^{t\Delta_K/2}$ and Kac's formula is \eqref{eq:kac-intro}.

Let $K\subset\R^n$ be a convex body with $0\in\intt K$.  For $T>0$ and $m\in\N$, put
\[
        h=\frac{T}{m},
        \qquad
        X^m=(X_h,X_{2h},\ldots,X_{mh})\in(\R^n)^m=\R^n\otimes\R^m.
\]
The law of $X^m$ is a centered Gaussian measure with covariance matrix
\[
        \Sigma_m=I_n\otimes \widetilde\Sigma_m,
        \qquad
        (\widetilde\Sigma_m)_{ij}=h\min(i,j).
\]
The inverse of $\widetilde\Sigma_m$ is
\begin{equation}\label{eq:inverse-covariance}
        \widetilde\Sigma_m^{-1}
        =\frac1{h}
        \begin{pmatrix}
        2 & -1 & 0 & \cdots & 0\\
        -1& 2 & -1& \cdots&0\\
        0&-1&2&\cdots&0\\
        \vdots&&&\ddots&-1\\
        0&0&0&-1&1
        \end{pmatrix}.
\end{equation}
In particular, if $R=(R_1,\ldots,R_m)\in(\R^n)^m$, then
\begin{equation}\label{eq:CM-norm}
        R^T\Sigma_m^{-1}R
        =\frac1{h}\left(|R_1|^2+
        \sum_{j=1}^{m-1}|R_{j+1}-R_j|^2\right).
\end{equation}

\subsubsection{From convexity to log-convexity}

We will use the following standard trick to deduce log-convexity from convexity by utilizing the homogeneity of $\lambda_1$.

\begin{lemma}\label{lem:convex_logconvex}
	Let $f:\R\to (0, \infty)$ have the property that $c^tf(t)$ is convex for any $c>0$. Then $f$ is log-convex.
\end{lemma}
\begin{proof}
	First assume $f\in C^2(\R)$.  It holds for any $b\in \R$ that $g(t)=e^{bt}f(t)$ satisfies $g''(t)\geq 0$. Thus
	\[
	f''+2bf'+b^2f\geq0
	\]
	for all $b\in \R$, and so $ff''\geq (f')^2$.  Thus $f$ is log-convex.
	
	In the general case, choose a non-negative cut-off function $\rho_\epsilon(t)\in C^\infty(\R)$. Then $f_\epsilon:=f\ast \rho_\epsilon$ converges pointwise to $f$ as $\epsilon\to 0$, and $f_\epsilon\in C^\infty(\R)$.
	
	Moreover, denoting $e_b(t)=e^{bt}$, it holds for all $b\in \R$ that
	
	\[e^{bt} f_\epsilon(t)=e^{bt}\int_{\R} f(t-s) \rho_\epsilon(s)ds=\int_\R e^{b(t-s)}f(t-s) e^{bs}\rho_\epsilon(s)ds = (e_bf)\ast (e_b\rho_\epsilon) (t). \]
	Thus $e^{bt} f_\epsilon(t)$ is convex for all $b\in\R$, and by the above $f_\epsilon$ is log-convex. It follows that $f=\lim f_\epsilon$ is log-convex.
	
\end{proof}

\subsection{The centrally symmetric case}\label{subsec:symmetric-logconvex}\label{sec:symmetric_log_convex}

We first prove Theorem \ref{thm:main-logconvex-intro} when $K$ is centrally symmetric. While a straightforward consequence of the B-theorem, it would fix a framework on top of which we will build in the general case. 

The following simple statement puts the strong B-theorem proved in \cite{BrascampLieb1975,CEFM} in invariant form.

\begin{lemma}\label{lem:b_conj}
	Let $\gamma$ be a centered Gaussian measure on $\R^d$ with covariance matrix $\Sigma$.  Assume that $S$ is a symmetric matrix and $S\Sigma^{-1/2}=\Sigma^{-1/2}S$. Then for all centrally symmetric convex bodies $A\subset\R^d$, 
	
	\[ \gamma(e^{tS}A)\geq \gamma(A)^{1-t}\gamma(e^SA)^t.\]
	
\end{lemma}
\proof
It holds by the B-theorem that for all centrally symmetric convex bodies $A$ in standard Gaussian space $(\R^d, \gamma_d)$ and all diagonal matrices $D$, that
$\gamma_d(e^{tD}A)$ is log-concave. Since any symmetric matrix $S=UDU^{-1}$ with $U$ orthogonal, it follows that
$\gamma_d(e^{tS}A)=\gamma_d(e^{tD}U^{-1}A)$ is log-concave for any such $S$.

Now we have 
\[ \gamma( e^{tS} A)=\gamma_d(\Sigma^{-1/2}e^{tS}A) =\gamma_d(e^{tS}\Sigma^{-1/2} A),\]
which by the above is log-concave.
\endproof

\begin{theorem}\label{thm:symmetric-logconvex}
Let $K\subset\R^n$ be centrally symmetric and convex, and let $P>0$.  Then $t\mapsto \lambda_1(P^tK)$ is log-convex.
\end{theorem}

\begin{proof}

First we prove for all $0<t<1$ that
\[ \lambda_1(P^tK)\leq (1-t)\lambda_1(K)+t\lambda_1(PK).\]

We may make an orthogonal change of coordinates such that $P$ is diagonal, and denote $P=\exp(D)$ where $D$ is a diagonal matrix with real entries.

Let $X_s$ be a standard Brownian motion in $\R^n$ starting at $0$. Denote $K_t=e^{tD}K$. Fix a time $\tau>0$ and an integer $m\geq 1$. Denote $s_j=j\tau/m$, $j=1, \dots, m$.  
Let $\tilde \gamma_m$ be the joint probability distribution of $(X_{s_1}, \dots, X_{s_m})\in \R^{nm}=\R^n\otimes\R^m$, which is Gaussian with covariance matrix $\Sigma_m=I_n\otimes \widetilde\Sigma_m$. Consider the event 
$A_t=\{ X_{s_1}\in K_t, \dots, X_{s_m}\in K_t\}$. Its probability is $\widetilde\gamma_m(A_t)$.  Geometrically $A_t= (K_t)^m= (e^{tD} K)^m= e^{tD_m}K^m$, where $D_m=D\otimes I_m$ is the diagonal matrix with $m$ copies of $D$ on the diagonal. Observe that $D_m$ commutes with $\Sigma_m^{-1/2}$.

We then have by Lemma \ref{lem:b_conj} that
\[\tilde \gamma_m(A_t)\geq \tilde\gamma_m(A_0)^{1-t}\tilde \gamma_m(A_1)^t.\]
Taking logarithms, we find

\[   \log\tilde\gamma_m(A_t)\geq (1-t)\log \tilde\gamma_m(A_0)+t\log \tilde\gamma_m(A_1)\]
that is,

\begin{align*} &\log \mathbb P_0(X_{s_1}, \dots, X_{s_m}\in P^t K)
	\\&\geq (1-t) \log \mathbb P_0(X_{s_1}, \dots, X_{s_m}\in K)+t\log \mathbb P_0(X_{s_1}, \dots, X_{s_m}\in PK).\end{align*}

Passing to the limit as $m\to\infty$ one has $\mathbb P_0(X_{s_1}, \dots, X_{s_m}\in \Omega)\to \mathbb P_0(\tau_\Omega>\tau)$, and we conclude that

\[\log \mathbb P_0 (\tau_{P^tK}>\tau)\geq (1-t)\log \mathbb P_0 (\tau_{K}>\tau)+t\log \mathbb P_0 (\tau_{PK}>\tau).\]

Dividing by $-\tau$ and letting $\tau\to\infty$ we finally conclude by Kac's formula that
\[ \lambda_1(P^tK)\leq (1-t)\lambda_1(K)+t\lambda_1(PK).\]

For arbitrary $t_0<t_1$ and $t=(1-s)t_0+st_1$, taking $K'=P^{t_0}K$, $P'=P^{t_1-t_0}$ shows that $\lambda_1(P^tK)\leq (1-s)\lambda_1(P^{t_0}K)+s\lambda_1(P^{t_1}K)$.

It follows that for every $c>0$,
\[
c^{2t}\lambda_1(P^tK)=\lambda_1((P/c)^tK)
\]
is convex. By Lemma \ref{lem:convex_logconvex}, $\lambda_1(P^tK)$ is log-convex.

\subsubsection{The Ornstein--Uhlenbeck operator}
\begin{proof}[Proof of Proposition \ref{prop:OU_convex}]
	The proof is the same as in Theorem \ref{thm:symmetric-logconvex}, with the Brownian process replaced with an Ornstein--Uhlenbeck process. The sampled joint probability distribution is Gaussian with a different covariance matrix of the same tensor product form $\Sigma_m'=I_n\otimes \widetilde \Sigma'_m$, and one may apply the B-theorem.
\end{proof}
	As the first Ornstein--Uhlenbeck eigenvalue is not homogeneous, convexity does not imply log-convexity.
\subsection{Discrete quasi-stationary distribution}\label{subsec:discrete-quasi}
\subsubsection{Discretely killed Brownian motion}
The non-symmetric proof requires a uniform estimate for the probability distribution of a Brownian motion killed only at discrete times, in the long time limit.
  
Let $K\subset\R^n$ be a convex body, and assume $0\in \operatorname{int}K$.  Let $(X_s)_{s\ge0}$ be Brownian motion in $\R^n$ with generator $\frac12\Delta$, started at $0$.  For $h>0$ put
\[
p_h(x,y)=(2\pi h)^{-n/2}\exp\left(-\frac{|x-y|^2}{2h}\right),
\]
and define the discrete time killed transition operator on $L^2(K)$ by
\[
Q_h f(x)=\int_K p_h(x,y)f(y)\dd y,\qquad x\in K.
\]
Equivalently, $Q_h=\1_K e^{h\Delta/2}\1_K$, where functions are extended by zero outside $K$. Let
\[
\rho_h=\rho_{1,h}>\rho_{2,h}\ge \rho_{3,h}\ge\cdots\ge0
\]
be the eigenvalues of $Q_h$, counted with multiplicity and ordered nonincreasingly.  Let $\varphi_h>0$ be its $L^2(K)$-normalized first eigenfunction:
\[
Q_h\varphi_h=\rho_h\varphi_h,\qquad \int_K\varphi_h^2=1.
\]
Define
\[
z_h:=\int_K x\varphi_h(x)^2\dd x.
\]
For $f$ defined on $K\subset \R^n$, we write $\widetilde f$ for the zero extension of $f$ to $\R^n$. We will write $\|f\|_2$ for the $L^2$ norm of $f$ when no confusion can arise concerning the domain. 

\subsubsection{Uniform discrete quasi-stationary mixing}
The following is the main result of section \ref{subsec:discrete-quasi}.

\begin{theorem}\label{thm:discrete_quasi}
	There are constants $C_K<\infty$, $c_K>0$, and $h_0(K)>0$ such that, for all $0<h<h_0(K)$ and all integers $1\le k\le m$,
	\[
	\left|\E_0\left[X_{kh}\,\middle|\, X_{jh}\in K,\ 1\le j\le m\right]-z_h\right|
	\le C_K\left(e^{-c_Kkh}+e^{-c_K(m-k)h}\right).
	\]
	Moreover, let $P$ be a symmetric positive definite matrix and set $K_t=P^tK$.  On every compact interval $I_0\subset\R$, the constants may be chosen uniformly in $t\in I_0$.
\end{theorem}

The continuous-time analogue of exponential convergence was studied by Champagnat et al in \cite{Champagnat_exponential, Champagnat_criteria, Champagnat_Q}, proving
\[
\left|\E_0[X_s\mid \tau_K>T]-z_K\right|
\le C_K\bigl(e^{-c_Ks}+e^{-c_K(T-s)}\bigr).
\]

\subsubsection{Approximating by the discrete time killed transition operator}

The operator $Q_h$ is compact, self-adjoint on $L^2(K)$ and positive semidefinite.  
Its kernel is strictly positive on the interior of $K\times K$, hence $Q_h$ is positivity improving.  By the Krein--Rutman theorem, $\rho_h$ is a simple eigenvalue, and the corresponding eigenfunction $\varphi_h$ can be chosen to be strictly positive.

Put
\[
\Pi_h f:=\ip{f}{\varphi_h}_{L^2(K)}\varphi_h,
\qquad
R_h:=Q_h-\rho_h\Pi_h.
\]
Then $R_h\varphi_h=0$, $R_h$ is self-adjoint, and for every $N\ge0$,
\begin{equation}\label{eq:RN}
	\norm{R_h^N f}_{2}\le \rho_{2,h}^N\norm{f}_2.
\end{equation}

We will need the following approximation result for the continuous killed Brownian motion by the discretely killed Brownian motion.
\begin{proposition}\label{prop:spectral-approx}
	Let $\mathcal G$ be a compact subset of the symmetric positive definite matrices.  For $G\in\mathcal G$ define
	\[
	p_{h,G}(z)=(2\pi h)^{-n/2}(\det G)^{-1/2}
	\exp\left(-\frac{\langle G^{-1}z,z\rangle}{2h}\right)
	\]
	and
	\[
	\cQ_{h,G}f(x)=\int_K p_{h,G}(x-y)f(y)\dd y.
	\]
	Let $\rho_{j,h}(G)$ be the eigenvalues of $\cQ_{h,G}$ in nonincreasing order and put
	\[
	\mu_{j,h}(G):=\frac{1-\rho_{j,h}(G)}{h}.
	\]
	Let $\alpha_j(G)$ be the $j$-th eigenvalue, counted with multiplicity, of the closed form
	\[
	\cE_G(u)=\frac12\int_K \langle G\nabla u,\nabla u\rangle\dd x,
	\qquad \operatorname{Dom}(\cE_G)=H^1_0(K).
	\]
	Then, for each fixed $j$,
	\[
	\sup_{G\in\mathcal G}|\mu_{j,h}(G)-\alpha_j(G)|\longrightarrow0
	\qquad\text{as }h\downarrow0.
	\]
	Moreover, let $\varphi_{h,G}>0$ be the  $L^2$-normalized first eigenfunction of $\cQ_{h,G}$, and $\psi_{1,G}>0$ the $L^2$-normalized first eigenfunction of $\cE_G$.  If $h_\ell \downarrow0$, then
	\[
	\varphi_{h_\ell ,G}\longrightarrow \psi_{1,G}
	\quad\text{in }L^2(K),
	\]
	uniformly in $G\in\mathcal G$.
	
\end{proposition}

\begin{proof}
	We use the unitary Fourier transform $u(x)\mapsto \widehat u(\xi)$ on $\R^n$.  For $u\in L^2(K)$, let $\widetilde u$ denote its zero extension.  The form associated with $(I-\cQ_{h,G})/h$ is
	\[
	\cE_{h,G}(u):=\frac1h\left(\norm{u}_2^2-\ip{u}{\cQ_{h,G}u}\right)
	=\int_{\R^n} m_{h,G}(\xi)|\widehat{\widetilde u}(\xi)|^2\dd\xi,
	\]
	where
	\[
	m_{h,G}(\xi)=\frac{1-e^{-h\langle G\xi,\xi\rangle/2}}{h},
	\]
	and its eigenvalues are $\mu_{j,h}(G)$.
	
	The min--max principle gives
	\[
	\mu_{j,h}(G)=\inf_{\dim V=j} \max_{\substack{u\in V\\ \norm{u}_2=1}} \cE_{h,G}(u),
	\]
		and similarly 
		\[
	\alpha_{j}(G)=\inf_{\dim V=j} \max_{\substack{u\in V\\ \norm{u}_2=1}} \cE_{G}(u),
	\]
	where in both cases one may consider only $V\subset H^1_0(K)$ as $\cE_{h,G}$ has full domain and $H^1_0(K)\subset L^2(K)$ is dense.
	
	First we prove the upper bound \begin{equation}\label{eq:limsup}
		\limsup_{h\downarrow0}\sup_{G\in\mathcal G}\bigl(\mu_{j,h}(G)-\alpha_j(G)\bigr)\le0.
	\end{equation}
	
	It holds for any $u\in H^1_0(K)$ that $\widetilde u\in H^1(\R^n)$, and so $|\xi|\widehat {\widetilde u}(\xi)\in L^2(\R^n)$. Thus we have
	
	\begin{align*}
		|\cE_{h,G}(u)- \cE_G(u)|&=\left|\int_{\R^n}(m_{h, G}(\xi)-\frac12 \langle G\xi, \xi\rangle)|\widehat {\widetilde u}(\xi)|^2\dd\xi \right|
		\\ &\leq\int_{\R^n}\left|m_{h, G}(\xi)-\frac12 \langle G\xi, \xi\rangle \right| |\widehat {\widetilde u}(\xi)|^2\dd\xi.
	\end{align*}
	
	Fix a finite-dimensional subspace $V\subset H^1_0(K)$, and let $\epsilon>0$. As
	
	\[
	0\le m_{h,G}(\xi)\le \frac12\langle G\xi,\xi\rangle\le C_{\mathcal G}|\xi|^2,
	\]
	
	we may choose $R>0$ such that for all $u\in V$ of unit $L^2$ norm, 
	\[\int_{|\xi|>R}\left|m_{h, G}(\xi)-\frac12 \langle G\xi, \xi\rangle \right| |\widehat {\widetilde u}(\xi)|^2\dd\xi\leq C_{\mathcal G}\int_{|\xi|>R}|\xi|^2 |\widehat {\widetilde u}(\xi)|^2<\epsilon. \]
	This is easily seen by fixing an $L^2$-orthonormal basis $u_i$ of $V\subset H^1_0(K)$, $1\leq i\leq N_V$, and writing $u=\sum_{i=1}^N c_i u_i$ with $\sum |c_i|^2=1$. 
	
	Next it holds that $m_{h,G}(\xi)\to \frac12\langle G\xi,\xi\rangle$ uniformly in  $|\xi|\leq R$, $G\in\mathcal G$. We conclude that \begin{equation}\label{eq:uniform}
		\cE_{h,G}(u)\longrightarrow \cE_G(u)
	\end{equation}
	uniformly in $G\in\mathcal G$ and  $u\in S_{L^2}(V)$, the $L^2$-unit sphere in $V$.
	
	Fix $\varepsilon>0$ and $G_0\in\mathcal G$.  Choose a $j$-dimensional subspace $V_{G_0}\subset H^1_0(K)$ such that
	\[
	\max_{u\in V_{G_0},\ \|u\|_2=1}\cE_{G_0}(u)\le \alpha_j(G_0)+\varepsilon.
	\]
	By the continuity of $G\mapsto \cE_G$ on the fixed finite-dimensional space $V_{G_0}$, as well as by the continuity of  $G\mapsto\alpha_j(G)$, there is a neighborhood $U_{G_0}$ of $G_0$ such that, for $G\in U_{G_0}$,
	\[
	\max_{u\in V_{G_0},\ \|u\|_2=1}\cE_G(u)\le \alpha_j(G)+2\varepsilon.
	\]
	By uniform convergence \eqref{eq:uniform}, it holds for all sufficiently small $h<h_0(V_{G_0})$, and uniformly in $G\in U_{G_0}$, that
	\[
	\max_{u\in V_{G_0},\ \|u\|_2=1}\cE_{h,G}(u)\le \alpha_j(G)+3\varepsilon.
	\]
	A finite subcover of the compact set $\mathcal G$ and the min--max principle therefore give
	\begin{equation*}
		\limsup_{h\downarrow0}\sup_{G\in\mathcal G}\bigl(\mu_{j,h}(G)-\alpha_j(G)\bigr)\le0.
	\end{equation*}
	as claimed. 
	
	\textit{Claim (precompactness).}  Let $h_\ell \downarrow0$, $\mathcal G\ni G_\ell \to G$, and suppose $u_\ell \in L^2(K)$ satisfies
	\[
	\norm{u_\ell }_2=1,
	\qquad
	\cE_{h_\ell ,G_\ell }(u_\ell )\le C.
	\]
	Then $\{u_\ell \}$ is precompact in $L^2(K)$. If $u_\ell \to u$ then $u\in H^1_0(K)$, and 
	\begin{equation}\label{eq:liminf-form}
		\cE_G(u)\le \liminf_{\ell \to\infty}\cE_{h_\ell ,G_\ell }(u_\ell ).
	\end{equation}
	
	\textit{Proof.}
	Since $\mathcal G$ is compact in the positive definite cone, there is $a>0$ such that
	\[
	\langle G_\ell \xi,\xi\rangle\ge a|\xi|^2
	\]
	for all $\ell $ and all $\xi$.  For every fixed $R>0$,
	\[
	|\xi|\ge R
	\quad\Longrightarrow\quad
	m_{h_\ell ,G_\ell }(\xi)\geq \frac{1}{h_\ell }(1-e^{-\frac {h_\ell } 2 aR^2})\to\frac12 aR^2.\]
	Hence for $\ell >\ell_0(R)$ and $|\xi|\geq R$, one has $m_{h_\ell ,G_\ell }(\xi)\geq \frac{aR^2}{4}$, and
	\[
	\cE_{h_\ell ,G_\ell}(u_\ell)=\int_{\R^n} m_{h_\ell ,G_\ell}(\xi)|\widehat{\widetilde u_\ell}(\xi)|^2\dd\xi\leq C\Rightarrow	\int_{|\xi|\ge R}|\widehat{\widetilde u_\ell }(\xi)|^2\dd\xi
	\le \frac{4C}{aR^2}
	\]
	for all $\ell >\ell_0(R)$.  
	
	We may now apply the Kolmogorov-Riesz theorem to deduce the precompactness of $\{\widetilde u_\ell \}$ in $L^2(\R^n$). Indeed, as $\{\widetilde u_\ell \}$ is bounded and uniformly supported in the fixed compact set $K$, all that remains is to verify mean equicontinuity.
	Take any $\epsilon>0$, choose $R$ such that $\frac{4C}{aR^2}<\epsilon$, and such that for all $\ell \leq \ell_0(R)$ it holds that $ \int_{|\xi|\ge R}|\widehat{\widetilde u_\ell }(\xi)|^2\dd\xi<\epsilon$. Therefore it holds for all $\ell $ that $ \int_{|\xi|\ge R}|\widehat{\widetilde u_\ell }(\xi)|^2\dd\xi<\epsilon$.
	Writing
	\[
	\|\widetilde u_\ell (\cdot+z)-\widetilde u_\ell \|_2^2
	=\int_{\R^n}|e^{iz\cdot\xi}-1|^2 |\widehat{\widetilde u_\ell }(\xi)|^2\dd\xi,
	\]
	the right-hand side is made small uniformly in $\ell $ by taking $|z|$ small enough such that for all $|\xi|\leq R$, $|e^{i\xi z}-1|^2<\epsilon$, implying
	\[
	\|\widetilde u_\ell (\cdot+z)-\widetilde u_\ell \|_2^2\leq \epsilon\int_{|\xi|\leq R}|\widehat{\widetilde u_\ell }(\xi)|^2\dd\xi   +4\int_{|\xi|\geq R}|\widehat{\widetilde u_\ell }(\xi)|^2\dd\xi \leq \epsilon+4\epsilon= 5\epsilon.\]
	Thus $\{u_\ell \}$ is precompact. Assume now $u_\ell \to u$ in norm in $L^2(K)$.
	
	For each fixed $R$, the convergence $G_\ell \to G$ and $h_\ell \downarrow0$ gives
	\[
	m_{h_\ell ,G_\ell }(\xi)\longrightarrow \frac12\langle G\xi,\xi\rangle
	\]
	uniformly on $|\xi|\le R$.  Since $\widehat{\widetilde u_\ell }\to\widehat{\widetilde u}$ in $L^2(\R^n)$, we obtain
	\[
	\int_{|\xi|\le R}\frac12\langle G\xi,\xi\rangle |\widehat{\widetilde u}(\xi)|^2\dd\xi
	\le \liminf_{\ell \to\infty}\cE_{h_\ell ,G_\ell }(u_\ell )
	\]
	for any $R>0$.

	Letting $R\to\infty$ shows $\widetilde u\in H^1(\R^n)$.  Since bounded convex domains are Lipschitz, a function $u\in L^2(K)$ whose zero extension belongs to $H^1(\R^n)$ must be in $H^1(K)$ with zero trace on $\partial K$, and hence $u\in H^1_0(K)$ \cite[Theorem 4.10]{Necas_elliptic}.
	We conclude that 
	\[
	\cE_G(u)\le \liminf_{\ell \to\infty}\cE_{h_\ell ,G_\ell }(u_\ell ),
	\] proving the claim.
	
Now let $h_\ell \downarrow0$, and $\mathcal G\ni G_\ell \to G$. Let $u_{1,\ell },\ldots,u_{j,\ell }$ be orthonormal eigenfunctions corresponding to the first $j$ eigenvalues
	\[
	\mu_{1,h_\ell }(G_\ell )\le\cdots\le\mu_{j,h_\ell }(G_\ell )
	\]
	of $(I-\cQ_{h_\ell ,G_\ell })/h_\ell $.  By eq. \eqref{eq:limsup}, $\mu_{j,h_\ell }(G_\ell )$ is bounded from above uniformly in $\ell $.  Applying the precompactness claim successively and passing to a subsequence, we may assume
	\[
	u_{i,\ell }\to u_i\quad\text{in }L^2(K),\qquad i=1,\ldots,j.
	\]
	The limits $u_i\in L^2(K)$ are orthonormal.  For a unit vector $a=(a_1,\ldots,a_j)\in\mathbb C^j$, put
	\[
	u_\ell ^a=\sum_{i=1}^j a_i u_{i,\ell },
	\qquad
	u^a=\sum_{i=1}^j a_i u_i.
	\]
	Then
	\[
	\cE_{h_\ell ,G_\ell }(u_\ell ^a)=\sum_{i=1}^j |a_i|^2\mu_{i,h_\ell }(G_\ell )
	\le \mu_{j,h_\ell }(G_\ell ).
	\]
	Using \eqref{eq:liminf-form}, we get
	\[
	\cE_G(u^a)\le \liminf_{\ell \to\infty}\mu_{j,h_\ell }(G_\ell ).
	\]
	Thus the $j$-dimensional space $U=\operatorname{span}\{u_1,\ldots,u_j\}$ satisfies
	\[
	\max_{\substack{u\in U\\ \norm{u}_2=1}}\cE_G(u)
	\le \liminf_{\ell \to\infty}\mu_{j,h_\ell }(G_\ell ).
	\]
	The min--max principle gives
	\[
	\alpha_j(G)\le \liminf_{\ell \to\infty}\mu_{j,h_\ell }(G_\ell ).
	\]
	Therefore by the continuity of $\alpha_j$, 
	
	\[
	\liminf_{\ell \to\infty}(\mu_{j,h_\ell }(G_\ell )-\alpha_j(G_\ell ))\geq 0.
	\]
			It follows that
	\[ \liminf_{h\downarrow0 }\inf_{G\in\mathcal {G}} (\mu_{j, h}(G)-\alpha_j(G))\geq 0.\]
	Together with \eqref{eq:limsup}, 
	we conclude that 
	\[ \mu_{j, h}(G)\to\alpha_j(G)\]
	as $h\to 0$, uniformly in $G\in\mathcal G$.
	
	Finally take $j=1$, $h_\ell  \to 0$, $G_\ell \to G$.  The same compactness argument shows that every $L^2$-limit $\psi$ of a convergent subsequence of $\varphi_{h_\ell , G_\ell }$ has $\mathcal E_G(\psi)\leq \alpha_1(G)$.  By the min-max principle, $\mathcal E_G(\psi)= \alpha_1(G)$. Since the first eigenvalue of the elliptic operator associated with $\cE_G$ is simple, $\psi$ must be the positive normalized first eigenfunction $\psi_{1,G}$.  Therefore the whole sequence $\varphi_{h_\ell , G_\ell }$ converges to $\psi_{1, G}$ in $L^2(K)$.
	
	Since $\psi_{1, G_\ell }\to \psi_{1, G}$ as $G_\ell \to G$ by the simplicity of the first eigenvalue, this immediately implies that $\varphi_{h_\ell , G}\to \psi_{1, G}$ as $\ell \to \infty$, uniformly in $G\in\mathcal G$.	
\end{proof}

\begin{lemma}\label{lem:spectral-ratio}
	For any $0<c_K<\frac{\lambda_2(K)-\lambda_1(K)}{2}$ there exist $h_0(K)>0$ such that for all  $0<h<h_0(K)$,
	\[
	\frac{\rho_{2,h}}{\rho_h}\le e^{-c_Kh}.
	\]
 Here $\lambda_j(K)$ are the Dirichlet eigenvalues of $-\Delta$ on $K$.
 
	Furthermore, if $P$ is positive definite and  $K_t=P^tK$, one may choose $c_{K_t}$ and $h_0(K_t)$ locally uniformly in $t$.
\end{lemma}

\begin{proof}
	For the fixed domain $K$, Proposition \ref{prop:spectral-approx} with $G=I$ gives
	\[
	\frac{1-\rho_{j,h}}h\longrightarrow \frac{\lambda_j(K)}2,
	\qquad j=1,2.
	\]
	Hence
	\[
	\log\frac{\rho_{2,h}}{\rho_h}
	=\log\left(1-h\frac{1-\rho_{2,h}}h\right)
	-\log\left(1-h\frac{1-\rho_h}h\right)
	=-\frac{\lambda_2(K)-\lambda_1(K)}2h+o(h),
	\]
	which implies the first assertion.
	
	Let us verify that $c_K, h_0(K)$ can be chosen locally uniformly in $t$ for $K_t=P^tK$.  Let $U_t:L^2(K)\to L^2(K_t)$ be the unitary map
	\[
	(U_tf)(P^tx)=(\det P^t)^{-1/2}f(x).
	\]
	Then
	\[
	U_t^{-1}Q_{h,t}U_t=\cQ_{h,G_t},
	\qquad
	G_t=P^{-2t}.
	\]
	On a compact $t$-interval, the matrices $G_t$ form a compact subset of the positive definite cone, so Proposition \ref{prop:spectral-approx} can be applied to deduce
	\begin{equation}\label{eq:uniform_convergence}
		\frac{1-\rho_{j,h}(G_t)}h\longrightarrow  \alpha_j(G_t)=\frac{\lambda_j(K_t)}2,
		\qquad j=1,2,
	\end{equation}
	locally uniformly in $t$. 
	
	As both $\lambda_1$ and $\lambda_2$ depend continuously on $K$ in the Hausdorff metric, while $\lambda_1$ is simple, the fundamental gap $\lambda_2-\lambda_1$ is locally bounded away from $0$ and so $c_K$ can be chosen locally uniformly.

	The locally uniform convergence in eq. \eqref{eq:uniform_convergence} allows to choose $h_0(K)$ locally uniformly as well.
\end{proof}

For the next proof we will need the Dirichlet heat kernel $p_{\tau}^{K,D}(x,y)$. It holds for all $\Omega\subset K$ that $\mathbb P_x(X_\tau\in \Omega, \tau_K>\tau)=\int_\Omega p_{\tau}^{K,D}(x,y) dy$, where $X_s$ is the Brownian motion originating at $x$. It holds that $p_{\tau}^{K,D}(x,y)$ is strictly positive in the interior of $K\times K$.

We denote by $q_h^{(N)}(x,y)$ the integral kernel of $Q_h^N$.

\begin{lemma}\label{lem:coeff}
	There exist $h_1(K)>0$, $a_K>0$  and $\Lambda_K<\infty$ such that, for all $0<h<h_1(K)$,
	\[
	\varphi_h(0)\ge a_K,
	\qquad
	\ip{\1}{\varphi_h}_{L^2(K)}\ge a_K, \qquad \rho_h\ge e^{-\Lambda_Kh}.
	\]
	Furthermore, if $K_t=P^tK$ then $h_1(K)$, $a_K$ and $\Lambda_K$ can be chosen locally uniformly in $t$.
\end{lemma}

\begin{proof}
	For fixed $K$, Proposition \ref{prop:spectral-approx} gives
	\[
	\varphi_h\to\psi_1\quad\text{in }L^2(K),
	\]
	where $\psi_1>0$ is the $L^2$-normalized first Dirichlet eigenfunction of $K$.  Thus
	\[
	\ip{\1}{\varphi_h}\to\ip{\1}{\psi_1}>0.
	\]
	
	Next we bound $\varphi_h(0)$ from below.  Choose $r>0$ such that $B(0,4r)\subset K$.  Since $\psi_1$ is positive in the interior of $K$, and $\varphi_h\to\psi_1$ in $L^2(K)$,
	\[
	\int_{B(0,r)}\varphi_h(y)\dd y\ge b_K>0
	\]
	for all sufficiently small $h$. For any such $h$, choose an integer $\ell=\ell(h)$ with $1\leq \ell h\leq2$.  By the eigenfunction equation,
	\[
	\varphi_h(0)=\rho_h^{-\ell}Q_h^\ell\varphi_h(0)
	=\rho_h^{-\ell}\int_K q_h^{(\ell)}(0,y)\varphi_h(y)\dd y
	\]
	 Since $0<\rho_h\le1$,
	\[
	\varphi_h(0)\ge \int_{B(0,r)}q_h^{(\ell)}(0,y)\varphi_h(y)\dd y.
	\]
	The discretely killed kernel in $K$ dominates the continuously killed Dirichlet heat kernel in $B(0,2r)$: indeed, the event that Brownian motion stays in $B(0,2r)$ for all times $0\le s\le\ell h$ is contained in the event that it lies in $K$ at the discrete times $h,2h,\ldots,\ell h$.  Therefore
	\[
	q_h^{(\ell)}(0,y)\ge p_{\ell h}^{B(0,2r),D}(0,y).
	\]
	Now  $p_{s}^{B(0,2r),D}(0, y)$ is strictly positive and continuous on $(s, y)\in [1,2]\times \overline{B(0,r)}$, hence bounded below by some $b_K'>0$.  Therefore
	\[
	\varphi_h(0)\ge b_K'\int_{B(0,r)}\varphi_h(y)\dd y\ge b_K'b_K>0.
	\]
	Finally, Proposition \ref{prop:spectral-approx} gives
	\[
	\rho_h=1-\frac{\lambda_1(K)}2h+o(h),
	\]
	so $\rho_h\ge e^{-\Lambda_Kh}$ for any $\Lambda_K>\frac12 \lambda_1(K)$, for all $h<h_1(K)$.
	
	For $K_t=P^tK$, let us verify that all constants are uniform on compact $t$-intervals $I_0$. Let $U_t:L^2(K)\to L^2(K_t)$ be the unitary map used in Lemma \ref{lem:spectral-ratio}.  Write
	\[
	\varphi_{h,t}=U_t\Phi_{h,t},
	\qquad
	\psi_{1,t}=U_t\Psi_{1,t},
	\]
	so that $\Phi_{h,t},\Psi_{1,t}\in L^2(K)$.  Proposition \ref{prop:spectral-approx}, applied to the compact matrix family $G_t=P^{-2t}$, implies
	\[
	\sup_{t\in I_0}\|\Phi_{h,t}-\Psi_{1,t}\|_{L^2(K)}\longrightarrow0.
	\]
	Since $t\mapsto \Psi_{1,t}$ is continuous in $L^2(K)$ and positive, compactness of $I_0$ gives
	\[
	\inf_{t\in I_0}\int_K \Psi_{1,t}(x)\dd x>0.
	\]
	Because $\det P^t$ is bounded above and below on $I_0$, this gives a uniform lower bound for
	\[
	\int_{K_t}\varphi_{h,t}(y)\dd y
	= (\det P^t)^{1/2}\int_K\Phi_{h,t}(x)\dd x.
	\]
	Similarly, $0\in\operatorname{int}K_t$ uniformly, so one may choose $r>0$ with $B(0,4r)\subset K_t$ for all $t\in I_0$.  The map
	\[
	t\longmapsto \int_{B(0,r)}\psi_{1,t}(y)\dd y
	= (\det P^t)^{1/2}\int_{P^{-t}B(0,r)}\Psi_{1,t}(x)\dd x
	\]
	is continuous and strictly positive, hence it has a positive minimum on $I_0$.  The uniform $L^2$ convergence above then yields
	\[
	\inf_{t\in I_0}\int_{B(0,r)}\varphi_{h,t}(y)\dd y>0
	\]
	for all sufficiently small $h$, and the heat kernel comparison is made in the same fixed ball $B(0,2r)$. Finally, the first Dirichlet eigenvalues of $P^tK$ are uniformly bounded on $I_0$.
\end{proof}

\subsubsection{Two $L^2$ remainder estimates}

\begin{lemma}\label{lem:l2}
	There exist constants $C_K<\infty$, $c_K>0$, and $T_0>0$ such that the following hold for all $0<h<h_1(K)$.
	\begin{enumerate}[label=\textup{(\alph*)}]
		\item For every $N\ge0$,
		\[
		\norm{Q_h^N\1-\rho_h^N\ip{\1}{\varphi_h}\varphi_h}_2
		\le C_K\rho_h^N e^{-c_KNh}.
		\]
		\item If $kh\ge T_0$, then
		\[
		\norm{q_h^{(k)}(0,\cdot)-\rho_h^k\varphi_h(0)\varphi_h}_2
		\le C_K\rho_h^k e^{-c_Kkh}.
		\]
	\end{enumerate}
	The constants may be chosen uniformly for $K_t=P^tK$ on compact $t$-intervals.
\end{lemma}

\begin{proof}
	Part (a) follows from \eqref{eq:RN} and Lemma \ref{lem:spectral-ratio}:
	\[
	\norm{Q_h^N\1-\rho_h^N\ip{\1}{\varphi_h}\varphi_h}_2
	=\norm{R_h^N\1}_2
	\le \rho_{2,h}^N\norm{\1}_2
	\le |K|^{1/2}\rho_h^Ne^{-c_KNh}.
	\]
	
	For part (b) let $\ell=\ell(h)$ be an integer such that
	\[
	\frac12\leq \ell h\leq1.
	\]
	For $kh\ge2$ we have $k >\ell$.  Put
	\[
	a_\ell(y):=q_h^{(\ell)}(0,y).
	\]
	The discretely killed kernel is bounded above by the free heat kernel, hence
	\[
	0\le a_\ell(y)\le p_{\ell h}(0,y).
	\]
	Since $\ell h\geq \frac12$, this gives
	\[
	\norm{a_\ell}_2\leq \|p_{\ell h}(0,y)\|_{2}\leq C'_{K}
	\]
	uniformly in $h$.  Now
	\[
	q_h^{(k)}(0,\cdot)=Q_h^{k-\ell}a_\ell.
	\]
	By self-adjointness and the eigenfunction equation,
	\[
	\ip{a_\ell}{\varphi_h}=Q_h^\ell\varphi_h(0)=\rho_h^\ell\varphi_h(0).
	\]
	Therefore
	\[
	q_h^{(k)}(0,\cdot)-\rho_h^k\varphi_h(0)\varphi_h=R_h^{k-\ell}a_\ell.
	\]
	Using eq. \eqref{eq:RN}, Lemma \ref{lem:spectral-ratio} and the lower bound $\rho_h\ge e^{-\Lambda_Kh}$ from Lemma \ref{lem:coeff}, we obtain
	\begin{align*}
	\norm{R_h^{k-\ell}a_\ell}_2&
	\leq C'_{K}\rho_h^{k-\ell}e^{-c_K(k-\ell)h}
	\leq C'_{K} e^{\Lambda_K \ell h+c_K \ell h}\rho_h^ke^{-c_Kkh}
	\\&\leq C'_{K}  e^{\Lambda_K+c_K} \rho_h^ke^{-c_Kkh}=C_K \rho_h^ke^{-c_Kkh},
	\end{align*}
	proving part (b) for $kh\ge2$, that is $T_0=2$.
	
	The proof is uniform in $t$ for $K_t=P^tK$ on compact intervals: Lemmas \ref{lem:spectral-ratio} and \ref{lem:coeff} are uniform, the volumes $|K_t|$ are uniformly bounded, and the free heat kernel $L^2$ bound is independent of the domain.
\end{proof}

\subsubsection{Proof of Theorem \ref{thm:discrete_quasi}}\label{sec:mixing}

Let
\[
E_{m,h}:=\{X_{jh}\in K\text{ for }1\le j\le m\}.
\]
The conditional probability density of $X_{kh}$ given $E_{m,h}$ is
\begin{equation}\label{eq:cond-density}
	\mu_{k,m}^h(y)
	=\frac{q_h^{(k)}(0,y)Q_h^{m-k}\1(y)}{Q_h^m\1(0)}.
\end{equation}
We will first assume $kh\geq T_K$ and $(m-k)h\ge T_K$, where $T_K>T_0$ is to be determined.  Put $N=m-k$,
\[
\eta_h:=\ip{\1}{\varphi_h},
\qquad
A:=\rho_h^k\varphi_h(0),
\qquad
B:=\rho_h^N\eta_h.
\]
By Lemma \ref{lem:l2},
\[
q_h^{(k)}(0,\cdot)=A\varphi_h+r_k,
\qquad
Q_h^N\1=B\varphi_h+s_N,
\]
with
\[
\norm{r_k}_2\le C_K\rho_h^ke^{-c_Kkh},
\qquad
\norm{s_N}_2\le C_K\rho_h^Ne^{-c_KNh}.
\]
Lemma \ref{lem:coeff} gives $\varphi_h(0)\ge a_K$ and $\eta_h\ge a_K$.  Hence by Cauchy-Schwarz, and since $\|\varphi_h\|_2=1$, 
\begin{align*}
	&\norm{q_h^{(k)}(0,\cdot)Q_h^N\1-AB\varphi_h^2}_{L^1(K)} \\
	&\qquad\le A\norm{s_N}_2+B\norm{r_k}_2+\norm{r_k}_2\norm{s_N}_2 \\
	&\qquad\le C_K'(ABe^{-c_KNh}\eta_h^{-1}+ ABe^{-c_Kkh}\varphi_h(0)^{-1}+ABe^{-c_K(k+N)h}),\\
	&\qquad\le C_K'' AB\left(e^{-c_Kkh}+e^{-c_KNh}\right).
\end{align*}
where the constant $C_K''$ depends only on $K$, and locally uniformly on $t$ for $K_t=P^tK$.

Therefore

\begin{align*} Q_h^m\1(0)=\int_K q_h^{(k)}(0,y)Q_h^N\1(y)dy=AB\int_K \varphi_h^2 + E=AB+E, \end{align*}
where
\[ |E|\leq C_K'' |K|AB\left(e^{-c_Kkh}+e^{-c_KNh}\right)=C_K'''AB(e^{-c_Kkh}+e^{-c_KNh}).\]

Choose $T_K>T_0$ large enough so that $e^{-c_Kkh}+e^{-c_KNh}<\frac{1}{2C_K'''}$ whenever $kh,Nh\ge T_K$, so that
\[ |E|=|Q_h^m\1(0)-AB|\leq \frac12 AB .\] 
We find that

\begin{align*}
	\norm{\mu_{k,m}^h-\varphi_h^2}_{L^1(K)}&= \frac{1}{AB+E} \| q_h^{(k)}(0,\cdot)Q_h^N\1 (y)-Q_h^m\1(0)\varphi_h^2\|_{L^1(K)}\\
	&\leq \frac{1}{\frac 12 AB}   \left( \norm{q_h^{(k)}(0,y)Q_h^N\1-AB\varphi_h^2}_{L^1(K)} + \|E \varphi_h^2\|_{L^1(K)}   \right) 
	\\ &\leq \frac{2}{AB} (C_K''AB(e^{-c_Kkh}+e^{-c_KNh})+ |E|)
\end{align*}
whenever $kh\ge T_K$ and $(m-k)h\ge T_K$, that is

\begin{equation}\label{eq:l1-main}
	\norm{\mu_{k,m}^h-\varphi_h^2(y)}_{L^1(K)}\leq  C_K'''(e^{-c_Kkh}+e^{-c_KNh}).
\end{equation}
Now if $kh\leq T_K$ or $Nh\leq T_K$, then the claimed upper bound is merely a constant depending on $K$. As	$\norm{\mu_{k,m}^h-\varphi_h^2}_{L^1(K)}\leq 2$, those cases can be accommodated by adjusting $C_K$. 

Finally, assuming $K$ lies in a ball of radius $R$, \eqref{eq:l1-main} implies
\[
\left|\E_0[X_{kh}\mid E_{m,h}]-\int_K y\varphi_h(y)^2\dd y\right|
\le R\,\norm{\mu_{k,m}^h-\varphi_h^2(y)}_{L^1(K)},
\]
proving Theorem \ref{thm:discrete_quasi} for fixed $K$.

The locally uniform bound for $K_t=P^tK$ is clear from the local uniformity of $R$ and the bounds provided by  Lemmas \ref{lem:coeff}, and \ref{lem:l2}.

\subsection{A discrete Kac formula}

We will need a diagonal version of the Kac formula, where we only condition on the Brownian motion staying in $K$ at discrete times $\frac{j}{m}T$, $1\leq j\leq m$, and simultaneously letting $T\to\infty$ while the resolution $h=\frac{T}{m}\to 0$.

\begin{proposition}\label{prop:discrete-kac}
Let $K\subset\R^n$ be a convex body and let $0\in\intt K$.  Suppose $T=T(m)=m^a$ with $0<a<1$.  Then
	\[
	\frac12\lambda_1(K)=
	-\lim_{m\to\infty}\frac1{T(m)}
	\log \mathbb P_0\{X_{jT/m}\in K\text{ for }1\le j\le m\}.
	\]
Moreover, the convergence is locally uniform in $t$ for $K_t=P^tK$.
\end{proposition}

\begin{proof}
	Denote $h=\frac{T}{m}=m^{a-1}$. We will use the notation and results of section \ref{sec:mixing}. It holds that
\[\mathbb P_0\{X_{jT/m}\in K\text{ for }1\le j\le m\}=Q_h^m\1(0), \]
We may assume $m$ is even, and choose $k=\frac{m}{2}$ so that $N=\frac{m}{2}$. Then
\[ Q_h^m\1(0) =AB(1+O(e^{-\frac 12 c_K m^a})  )\sim  AB,\quad m\to\infty .\]
It holds that
\[ AB=\rho_h^m\varphi_h(0)\eta_h, \]
and by Proposition  \ref{prop:spectral-approx}, $\eta_h=\int _K\varphi_h$ is bounded from above as $h\to 0$. Similarly, $\varphi_h(0)$ is bounded from above as $h\to 0$.
Indeed, choose $\ell$ with $\frac12\leq \ell h\leq1$. Then as in the proof of Lemma \ref{lem:l2},
\[
\varphi_h(0)
=\rho_h^{-\ell}\int_K q_h^{(\ell)}(0,y)\varphi_h(y),dy
\leq \rho_h^{-\ell}
\|p_{\ell h}(0,\cdot)\|_2\|\varphi_h\|_2= \rho_h^{-\ell}
\|p_{\ell h}(0,\cdot)\|_2,
\]
and using the bound $\rho_h^{-\ell}\leq e^{\Lambda_K\ell h}\leq e^{\Lambda _K}$ provided by Lemma \ref{lem:coeff} we conclude that 
\[ \varphi_h(0)\leq C_K.\]

Consequently, 

\[ \frac{1}{T} \log  Q_h^m\1(0)  =\frac{m}{T}\log \rho_h + O(\frac{1}{T}).\]
By eq. \eqref{eq:uniform_convergence},
	\begin{equation*}
	\frac{1-\rho_{h}}h\to \frac{\lambda_1(K)}2,
	\end{equation*}
	and so 
	\[\frac{m}{T}\log \rho_h \sim h^{-1} \log(1-h\frac12\lambda_1(K))\sim -\frac12 \lambda_1(K),\]
	proving the claimed formula. 
	
	The locally uniform in $t$ convergence follows from the corresponding statements of Proposition \ref{prop:spectral-approx} and Lemma  \ref{lem:coeff}.
\end{proof}

\subsection{The $L^2$ estimate}\label{subsec:L2-estimate}

The following is the $L^2$ estimate used in the proof of the B-theorem \cite{CEFM}.

\begin{lemma}\label{lem:L2}
	Let $\nu$ be a smooth probability measure on $\R^N$ given by  $d\nu=e^{-W}dx$ with $\operatorname{Hess}W\geq I_N$. Let $q\in H^1(\R^N, \nu)$ satisfy $\int q(x)d\nu(x)=0$. Then
	\begin{equation}\label{eq:L2-estimate}
		\int q^2\,d\nu-\frac12\int |\nabla q|^2\,d\nu\leq \left|\int\nabla q\,d\nu\right|^2.
	\end{equation}
\end{lemma}

\begin{proof}

	Let $L=\Delta-\langle\nabla W,\nabla\rangle$.  The range of $L$ on $C_c^\infty(\R^N)$ is dense in the subspace of $L^2(\nu)$ of mean-zero functions.  Hence it suffices to prove, for all $u\in C_c^\infty(\R^N)$,
	\[
	\int\left((Lu-q)^2-q^2+\frac12|\nabla q|^2\right)d\nu
	\geq
	-\left|\int\nabla q\,d\nu\right|^2.
	\]
	Using integration by parts,
	\[
	-\int (Lu)q\,d\nu=\int\langle\nabla u,\nabla q\rangle d\nu.
	\]
	The Bochner identity for $L$ gives
	\[
	\int (Lu)^2d\nu
	=\int\left(\|\operatorname{Hess}u\|_{\HS}^2
	+\langle \operatorname{Hess}W\nabla u,\nabla u\rangle\right)d\nu
	\ge
	\int\left(\|\operatorname{Hess}u\|_{\HS}^2+|\nabla u|^2\right)d\nu.
	\]
	Thus it is enough to show
	\[
	\int\left(\|\operatorname{Hess}u\|_{\HS}^2+|\nabla u|^2
	+2\langle\nabla u,\nabla q\rangle+\frac12|\nabla q|^2\right)d\nu\geq	-\left|\int\nabla q\,d\nu\right|^2.
	\]
	Let $c=\int\nabla u\,d\nu$ and write $\nabla u=\nabla u_0+c$, where $\int\nabla u_0\,d\nu=0$.

	The last inequality can be rewritten as
	
	\[ |c|^2+ \int\left( \|\mathrm{Hess}u_0\|_{\HS}^2-|\nabla u_0|^2+\frac 12 |2\nabla u_0+\nabla q|^2\right)d\nu +2\left\langle c, \int \nabla q d\nu\right\rangle\geq	-\left|\int\nabla q\,d\nu\right|^2.  \]

	The Gaussian Poincar\'e inequality for measures with $\operatorname{Hess}W\ge I$ \cite{BrascampLieb} applied to each coordinate of the mean-zero vector $\nabla u_0$ implies
	\[
	\int\left(\|\operatorname{Hess}u_0\|_{\HS}^2-|\nabla u_0|^2\right)d\nu\ge0,
	\]
	and clearly
	\[
	|c|^2+2\left\langle c,\int\nabla q\,d\nu\right\rangle= \left|c+\int\nabla q\,d\nu \right|^2-\left|\int\nabla q\,d\nu\right|^2
	\ge
	-\left|\int\nabla q\,d\nu\right|^2,
	\]
	concluding the proof.
\end{proof}

We now apply this bound to the restricted Gaussian measure.

\begin{lemma}\label{lem:L2_Applied}
	Let $\gamma_N$ be the standard Gaussian measure on $\R^N$. Let $A\subset\R^N$ be convex and let
	\[
	F(t)=\log\gamma_N(e^{tS}A),
	\]
	where $S$ is symmetric.  Let $d\mu_t =e^{-F(t)}\1_{e^{tS}A} d\gamma_N$ be $\gamma_N$ restricted to $e^{tS}A$, normalized to be a probability measure.
	Then 
		\begin{equation}\label{eq:L2_bound}F''(t)\leq  4\left|S\int x\,d\mu_t(x)\right|^2.
	\end{equation}
	\end{lemma} 
	
	\begin{proof}
 Put
\[
Q(x)=\langle Sx,x\rangle,
\qquad
q(x)=Q(x)-\int Q\,d\mu_t, \qquad \nabla q(x)=2 Sx.
\]
A direct differentiation gives
\begin{equation}\label{eq:F-second-derivative}
        F''(t)=\int q^2\,d\mu_t-\frac12\int |\nabla q|^2\,d\mu_t.
\end{equation}
	Approximate $\mu_t$ by smooth probability measures $d\nu_j=e^{-W_j}dx$ as follows. Set $d_t(x)=\mathrm{dist}(x, e^{tS}A)$, let $\eta(x)\in C^\infty(\R^N)$ be non-negative, compactly supported in the unit ball, with $\int\eta=1$. Denote $\eta_\epsilon(x)=\epsilon^{-N}\eta(x/\epsilon)$.
	Define
	\[W_j(x)=\frac12 |x|^2+j(d_t^2 \ast \eta_{1/j}) (x)+c_j,\]
	where $c_j$ is chosen such that $\int e^{-W_j}=1$.	Then $ \operatorname{Hess}W_j\geq I_N$, and $\int p(x) d\nu_j\to \int p(x)d\mu_t$ for any polynomial $p$.
	
	Letting $q_j:=Q-\int Q\,d\nu_j$ and applying Lemma \ref{lem:L2} yields
	
	\[ \int q_j^2\,d\nu_j-\frac12\int |\nabla q_j|^2\,d\nu_j\leq  \left|\int\nabla q_j\,d\nu_j\right|^2=4\left|S\int x\,d\nu_j(x)\right|^2.\]
	In the limit $j\to\infty$ we arrive, using eq. \eqref{eq:F-second-derivative}, at the claimed bound. 
\end{proof}
\subsection{Proof of log-convexity in Theorem \ref{thm:main-logconvex-intro}}\label{sec:proof-logconvex}

We now prove that $\lambda_1(P^tK)$ is log-convex.  Translate $K$ to have $0$ in its interior. By an orthogonal change of variables, write $P=e^D$ with $D$ diagonal.  Fix $M>0$ and consider $t\in[-M,M]$.  Put
\[
        K_t=P^tK=e^{tD}K.
\]
Let $T=T(m)=m^{1/\zeta}$, where $\zeta>1$, and let $h=T/m$.  Define
\[
        f_m(t)=\log\mathbb P_0\{X_h,X_{2h},\ldots,X_{mh}\in K_t\}.
\]
Let $\Sigma_m$ be the covariance matrix of $X^m$.  Since $D_m=D\otimes I_m$ commutes with $\Sigma_m$, we may use the standard Gaussian measure to write
\[
        f_m(t)=\log\gamma_{nm}(e^{tD_m}\Sigma_m^{-1/2}K^m).
\]
Let $\mu_{t,m}$ be the normalized Gaussian measure on
\[
        e^{tD_m}\Sigma_m^{-1/2}K^m=
        \Sigma_m^{-1/2}K_t^m.
\]
By Lemma \ref{lem:L2_Applied},
\begin{equation}\label{eq:fm-second-bound}
        f_m''(t)\le 4 \left|D_m\int x\,d\mu_{t,m}(x)\right|^2\leq C  \left|\int x\,d\mu_{t,m}(x)\right|^2.
\end{equation}
Changing variables back to the coordinates of the Brownian motion gives
\begin{equation}\label{eq:gradq-R}
        \int x\,d\mu_{t,m}(x)
        =\Sigma_m^{-1/2}R(t),
\end{equation}
where
\[
        R(t)=\frac{1}{\mathbb P_0(X^m\in K_t^m)}
        \int_{K_t^m} y\,d\widetilde\gamma_m(y)
        =\mathbb E_0[X^m\mid X^m\in K_t^m].
\]
Write $R(t)=(R_1(t),\ldots,R_m(t))$.

Recall the discrete time killed transition operator $Q_h=Q_{h, K_t}$ and its first $L^2$-normalized positive eigenfunction $\varphi_h=\varphi_{h, K_t}$ defined in section \ref{subsec:discrete-quasi}.
Write
\[
z_{h,t}=\int_{K_t} x\varphi_{h, K_t}(x)^2\dd x,
\]
and note that $z_{h,t}\in K_t\subset\R^n$ and in particular uniformly bounded in $|t|\leq M$. 
Let $\eps=T^{-\alpha}$ with $0<\alpha<1$. By Theorem \ref{thm:discrete_quasi}, it holds uniformly in $|t|\le M$ that
\[
        R_j(t)=z_{h,t}+O(e^{-c_M\eps T}),
\]
whenever $\eps m\le j\le(1-\eps)m$. Thus we may write
\[
        R(t)=U_t+V_t \in (\R^n)^m,\]
    where $U_t=(z_{h,t},\dots, z_{h, t})$, while $V_t$ has $O(1)$ entries in the first and last $\eps m$ positions and $O(e^{-c\eps T})$ entries in the remaining $(1-2\epsilon)m$ positions.

Using the inverse covariance \eqref{eq:CM-norm} we find 
\begin{equation}\label{eq:U-estimate}
        |\Sigma_m^{-1/2}U_t|^2
        =U_t^T\Sigma_m^{-1}U_t
        \le C_M\frac{m}{T}.
\end{equation}
Similarly,
\begin{equation}\label{eq:V-estimate}
        |\Sigma_m^{-1/2}V_t|^2=V_t^T\Sigma_m^{-1}V_t
        \le
        C_M\left(\eps\frac{m^2}{T}
        +\frac{m^2}{T}e^{-c_MT^{1-\alpha}}\right).
\end{equation}
Hence
\begin{equation}\label{eq:gradq-estimate}
        \left|\int x \,d\mu_{t,m}(x)\right|
        \le
        C_M\left(\sqrt{\frac{m}{T}}
        +\sqrt\eps\frac{m}{\sqrt T}
        +\frac{m}{\sqrt T}e^{-c_MT^{1-\alpha}/2}\right).
\end{equation}
Now set
\[
        \zeta=1+\delta,
        \qquad
        \alpha=1-\delta,
\]
with $0<\delta<1/3$.  Since $m=T^\zeta$, the dominant term in \eqref{eq:gradq-estimate} is 
\[\sqrt\eps\frac{m}{\sqrt T}=O(T^{3\delta/2}).\]
Therefore
\begin{equation}\label{eq:betaT}
        f_m''(t)\le \beta(T):=C_MT^{3\delta}=o(T)
\end{equation}
uniformly for $t\in[-M,M]$.

It follows that
\[
        -\frac1T f_m(t)+\frac{\beta(T)}{2T}t^2
\]
is convex on $[-M,M]$.  By Proposition \ref{prop:discrete-kac},
\[
        -\frac1T f_m(t)\longrightarrow \frac12\lambda_1(K_t)
\]
pointwise as $m\to\infty$, and $\beta(T)/T\to0$. As convexity is preserved under pointwise limits in intervals, 
\[
        t\longmapsto \lambda_1(P^tK)
\]
is convex on $[-M,M]$.  Since $M$ is arbitrary, $\lambda_1(P^tK)$ is convex on $\R$.

As in Theorem \ref{thm:symmetric-logconvex}, it now follows by Lemma \ref{lem:convex_logconvex} that $ \lambda_1(P^tK)$ is log-convex.
\end{proof}

\subsection{Strict log-convexity}
To complete the proof of Theorem \ref{thm:main-logconvex-intro}, it remains to show that $\lambda_1(P^tK)$ is strictly log-convex for $\SL(n)\ni P\neq I_n$. This will follow from the following two statements. 
The first asserts that $\lambda_1(gK)$ is proper.

\begin{proposition}\label{prop:minimum_existence}
	For any convex body $K\subset\R^n$, the function $\SL(n)\ni g\mapsto \lambda_1(gK)$ satisfies $\lim_{g\to \infty} \lambda_1(gK)=\infty$.
\end{proposition}  
\proof
Assume $0$ is the centroid of $K$. Denoting by $r(K)$ the inradius, we will show $r(gK)\to 0$ as $\SL(n)\ni g\to \infty$.

Now if $g_j\in\SL(n)$ approaches $\infty$, there is an eigenvalue $\mu_j\to \infty$ of $g_j^Tg_j$.
We may find corresponding eigenvectors $v_j^\pm\in\partial K$ such that $v^-_j=-\alpha v^+_j$ with $\alpha_j>0$. It then holds that $|g_jv^\pm_j|^2=\mu_j|v^\pm_j|^2\to \infty$. 

Assume in the way of contradiction there is $\epsilon>0$ such that $r_j=r(g_jK)\geq \epsilon$ for all $j$. Let $B_j$ be the inscribed ball of $g_jK$,  and consider an affine hyperplane $H$ through the center of $B_j$ which is perpendicular to the line $\mathrm{Span}(g_jv^\pm_j)$. At least one of the points $g_jv^\pm_j$, say $g_jv^+_j$, has distance at least $\frac12 |g_jv^+_j-g_jv^-_j|$ from $H$.
Then $ C_j:=\mathrm{Conv}(g_jv^+_j, H\cap B_j)\subset g_jK$ while \[\vol(C_j)=\frac{1}{n}r_j^{n-1}|B^{n-1}| \mathrm{dist}(g_jv_j^+, H)\geq c_n\epsilon^{n-1}|g_jv^+_j-g_jv^-_j| \to \infty,\] in contradiction to $\vol(g_jK)=\vol(K)$. 

It follows from Payne--Stakgold \cite{Payne_Stakgold}, see also Protter \cite{Protter}, that $\lambda_1(K)\geq \frac{\pi^2}{4}\frac{1}{r(K)^2}$, and so  $\lambda_1(gK)\to \infty$ as $\SL(n)\ni g\to \infty$. 
\endproof

The second lemma is a standard result from perturbation theory.

\begin{lemma}\label{lem:analytic}
	For any convex body $K$ and $P$ positive definite, $t\mapsto \lambda_1(P^tK)$ is real analytic in $t\in \R$.
\end{lemma}
\proof

Pull the Dirichlet problem on $P^tK$ back to $K$.  It is associated with the
closed form
\[
  \mathfrak a_t(u,v)
  =\int_K\langle P^{-2t}\nabla u,\nabla\overline v\rangle\,dx,
  \qquad
  \operatorname{Dom}(\mathfrak a_t)=H^1_0(K).
\]
The form domain is independent of $t$, and 
it is easy to check that $\mathfrak a_t(u,v)$ is a  holomorphic family of forms in the sense of Kato; see
\cite[Chapter~VII]{Kato}.  Since $K$ is bounded, the associated
operators have compact resolvent.  The first eigenvalue is simple and isolated,
and therefore Kato's analytic perturbation theory implies that
\[
  t\longmapsto\lambda_1(P^tK)
\]
is real analytic in a neighborhood of zero.   Replacing $K$ by $P^tK$ yields analyticity for all $t\in\R$.
\endproof

Now if $\SL(n)\ni P\neq I_n$ then $P^t\to \infty$ at $t\to \pm \infty$. Consequently, $\lim_{t\to\pm \infty}\lambda_1(P^tK)=\infty$ by Proposition \ref{prop:minimum_existence}. If $\log \lambda_1(P^tK)$ is affine in an interval then by analyticity it must be affine everywhere, and so it is impossible that $\lim_{t\to\pm\infty}\log \lambda_1(P^tK)=\infty$. This contradiction shows $\log \lambda_1(P^tK)$ is strictly log-convex.

\subsection{Faber--Krahn position}

\begin{proof}[Proof of Theorem \ref{thm:FK-position}]
The existence follows from Proposition \ref{prop:minimum_existence}, together with the continuity of $\lambda_1$ with respect to the Hausdorff metric on convex bodies.

For uniqueness, assume that $K_0$ and $K_1=TK_0$ both minimize $\lambda_1$ on  $\mathrm{SL}(n)K$, with $\lambda_1(K_0)=\lambda_1(K_1)$ and $T$ not orthogonal.
Let $T=BP$ be the polar decomposition with $B$ orthogonal and $P$ positive definite, $\det P=1$, $P\neq I_n$. We then have $\lambda_1(TK_0)=\lambda_1(P K_0)$.
But this contradicts the strict log-convexity of $\lambda_1(P^tK)$ asserted in \ref{thm:main-logconvex-intro}.
\end{proof}

\begin{proof}[Proof of Proposition \ref{prop:EL}]
Let $S$ be symmetric with $\tr S=0$, and put $A_t=e^{tS}$.  Pulling the Rayleigh quotient on $A_tK$ back to $K$ gives
\[
        \lambda_1(A_tK)=
        \inf_{v\in H^1_0(K)\setminus\{0\}}
        \frac{\int_K \langle e^{-2tS}\nabla v,\nabla v\rangle\,dx}{\int_K v^2\,dx}.
\]
Differentiating at $t=0$ and using simplicity of the first eigenvalue gives
\[
        \left.\frac{d}{dt}\right|_{t=0}\lambda_1(A_tK)
        =-2\int_K\langle S\nabla \psi,\nabla \psi\rangle\,dx.
\]
At a minimizer this derivative vanishes for every trace-free symmetric $S$, which is equivalent to \eqref{eq:EL-energy}.  
This establishes the necessary condition.

To see that it is sufficient, observe that it implies that $t=0$ is a critical point of $\lambda_1(P^tK)$ for all $P>0$, $P\neq I_n$ with $\det P=1$. As this function is strictly log-convex, a critical point can only correspond to a global minimum. As this is true for all $P$, $K$ must be in Faber--Krahn position.
\end{proof}

\begin{proposition}\label{prop:FK_continuity}
	The function $K\mapsto \Lambda_{\FK}(K)$ is continuous in the Hausdorff metric.
\end{proposition}
\begin{proof}
	Assume $K_j\to K$. We ought to show that $\Lambda_{\FK}(K_j)$ $\to$ $\Lambda_{\FK}(K)$. 
	Applying a special linear transformation to all bodies, we may assume $\Lambda_{\FK}(K)=\lambda_1(K)$. It then holds that $\lambda_1(K_j)\to \lambda_1(K)=\Lambda_{\FK}(K)$.
	
	We will show that $\lambda_1(K_j)-\Lambda_{\FK}(K_j)\to 0$. If not, we may pass to a subsequence and find $\epsilon>0$ such that 
	\[ \lambda_1(K_j)\geq \Lambda_{\FK}(K_j)+\epsilon.\]
	Thus there exist $h_j\in\SL(n)$ such that  $\lambda_1(h_jK_j)\leq \lambda_1(K_j)-\epsilon$.
	
	By assumption, there is a sequence $\epsilon_j\to 0$ such that $K_j\subset K+\epsilon_j B$, $K\subset K_j+\epsilon_j B$, where $B$ is the unit ball. In particular, by monotonicity it holds that
	\begin{equation}\label{eq:sandwich} \lambda_1( h_jK+\epsilon_j h_jB)\leq \lambda_1(h_j K_j)\leq \lambda_1(K_j)-\epsilon\to \lambda_1(K)-\epsilon.\end{equation}
	If $h_j\in SL(n)$ is not bounded, we could pass to a subsequence $h_{j}\to \infty$. Assuming $j$ large enough so that $\epsilon_j\leq 1$, we find by Proposition \ref{prop:minimum_existence}
	
	\[ \lambda_1(h_jK_j)\geq \lambda_1(h_j(K+ B))\to \infty, \]
	in contradiction to eq. \eqref{eq:sandwich}.
	
	Thus $h_j$ is bounded, and so $| \lambda_1( h_jK+\epsilon_j h_jB)-\lambda_1(h_jK)| \to 0$ as $j\to\infty$. 
	But then by eq. \eqref{eq:sandwich} we find that for large $j$, $\lambda_1(h_jK)\leq \lambda_1(K)-\epsilon/2$, contradicting our assumption that $K$ is in Faber--Krahn position. This contradiction completes the proof.

\end{proof}

\subsection{Log-convexity of inverse inradius and Cheeger constant}\label{subsec:inradius-cheeger-proofs}

We write $K\ominus E=\{x:x+E\subset K\}$ for the Minkowski erosion, and $B$ for the unit Euclidean ball of the same dimension as  $K$.

\begin{proof}[Proof of Proposition \ref{prop:inradius-logconvex}]Let $P=e^D$, and suppose $r(K_0)=r_0$ and $r(K_1)=r_1$ for $K_i=P^iK$, $i=0,1$.  Then there exist $x_0,x_1$ such that
\[
        x_0+r_0B\subset K,
        \qquad
        x_1+r_1B\subset PK.
\]
Equivalently,
\[
        x_0+r_0B\subset K,
        \qquad
        P^{-1}x_1+r_1P^{-1}B\subset K.
\]
For $s\in[0,1]$, convexity of $K$ implies that
\begin{equation}\label{eq:two_balls}
        (1-s)x_0+sP^{-1}x_1+(1-s)r_0B+s r_1P^{-1}B\subset K.
\end{equation}
Now the two positive matrices $r_0^{1-s}r_1^sP^{-s},(1-s)r_0I+sr_1P^{-1}$ can be simultaneously diagonalized, and the AM-GM inequality shows $r_0^{1-s}r_1^sP^{-s}\leq (1-s)r_0I+sr_1P^{-1}$. It follows that

\[
        r_0^{1-s}r_1^sP^{-s}B
        \subset
         ((1-s)r_0I+sr_1P^{-1})B
        \subset 
         (1-s)r_0B+sr_1P^{-1}B.
\]
By eq. \eqref{eq:two_balls},
\[   (1-s)x_0+sP^{-1}x_1+r_0^{1-s}r_1^sP^{-s}B\subset K.\]
Applying $P^s$ we conclude 
\[ P^s((1-s)x_0+sP^{-1}x_1)+ r_0^{1-s}r_1^s B\subset P^sK,\]
and so $r(P^sK)\geq r_0^{1-s}r_1^s$. This proves log-convexity of $1/r$.
\end{proof}

\begin{proof}[Proof of Proposition \ref{prop:cheeger-logconvex}]
For a planar convex body $\Omega$, let $\Omega^r=\Omega\ominus rB$.  The Cheeger radius $\rcheb(\Omega):=1/h(\Omega)$ is characterized by
\begin{equation}\label{eq:cheeger-radius-char}
        |\Omega^{\rcheb(\Omega)}|=\pi\rcheb(\Omega)^2,
\end{equation}
see the characterization of Cheeger sets in convex planar domains \cite{KawohlLachandRobert}.  

Assume first $\det P=1$. Let $K_i=P^iK$ and set $r_i=\rcheb(K_i)$, $i=0,1$.  Pulling back to $K$, equation \eqref{eq:cheeger-radius-char} says
\[
        |K\ominus r_0B|=\pi r_0^2,
        \qquad
        |K\ominus r_1P^{-1}B|=\pi r_1^2.
\]
Put $r_s=r_0^{1-s}r_1^s$.  As in the proof of Proposition \ref{prop:inradius-logconvex},
\[
        r_sP^{-s}B\subset (1-s)r_0B+sr_1P^{-1}B.
\]
Hence
\[
        K\ominus r_sP^{-s}B
        \supset
        K\ominus\big((1-s)r_0B+sr_1P^{-1}B\big)
        \supset
        (1-s)(K\ominus r_0B)+s(K\ominus r_1P^{-1}B).
\]
By the planar Brunn--Minkowski inequality,

\begin{align*}
        |K\ominus r_sP^{-s}B|^{1/2}
        &\geq
        (1-s)|K\ominus r_0B|^{1/2}
        +s|K\ominus r_1P^{-1}B|^{1/2} \\
        &=\sqrt\pi\big((1-s)r_0+sr_1\big)
        \\&\geq \sqrt\pi\, r_s.
\end{align*}

After applying $P^s$, we get $|(P^sK)^{r_s}|\ge \pi r_s^2$.  By the characterization \eqref{eq:cheeger-radius-char}, $\rcheb(P^sK)\ge r_s$.  Hence $h(P^sK)\le h(K)^{1-s}h(PK)^s$, concluding the proof for $\det P=1$.

For general $P$, put $P=cP_1$ with $c>0$ and $\det P_1=1$. Then $h(P^tK)=c^{-t}h(P_1^tK)$ is log-convex by the previous argument.
\end{proof}

\subsection{Concavity of the heat trace}

\begin{proof}[Proof of Theorem \ref{thm:heattrace}]
	We make use of the notation in the proof of Theorem \ref{thm:symmetric-logconvex}.
The heat trace has the Brownian-loop representation
\[
        Z_K(\tau)=\int_K p^{K, D}_{\tau}(x,x)\,dx,
\]
We may discretize the Brownian loops as follows:
\[
        Z_K(\tau)=\lim_{m\to\infty}\left(\frac{m}{2\pi \tau}\right)^{nm/2}
        \int_{K^m}
        \exp\left(-\frac{m}{2\tau}
        \sum_{j=1}^m|x_{j+1}-x_j|^2\right)
        dx_1\cdots dx_m,
\]
where $x_{m+1}=x_1$.  The quadratic form in the exponent is degenerate only in the constant direction.  One regularizes it by adding $\eta\sum|x_j|^2$. We then apply the Gaussian B-theorem to the centrally symmetric convex set $K^m$, and let $\eta\downarrow0$.  Since $(P^tK)^m=e^{tD_m}K^m$, and the quadratic form commutes with $D_m$, this gives log-concavity of $t\mapsto Z_{P^tK}(\tau)$.

Now assume $K$ has an orthogonal symmetry group $G(K)$ acting irreducibly on $\R^n$, and let \[F_K(S)=\left.\frac{d}{dt}\right|_{t=0}Z_{e^{tS}K}(\tau).\] 
Assume $\mathrm{tr}(S)=0$. Then $S_0:=\int_{G(K)}gSg^{-1}dg$ must be a scalar matrix $cI_n$ by Schur's lemma, and it holds that 
\[nc=\mathrm{tr}(S_0)=\mathrm{tr}(S)=0.\] Consequently, $S_0=0$, and
\[F_K(S)=\int_G F_K(g^{-1}Sg)dg=F_K(S_0)=0. \]
 A log-concave function with zero derivative at $0$ in every trace-free direction is maximized at $0$ along every positive definite $\SL(n)$ flow.  The polar decomposition gives the result for all $A\in \SL(n)$.
\end{proof}

\section{Schmuckenschl\"ager-type inequalities}

\subsection{The improved inequality}

The Gaussian conjugate Rogers-Shephard inequality of \cite{Milman_Nakamura_Tsuji} implies 
\begin{equation}\label{eq:improved-schmuck-intro}
	\lambda_1(K\cap L)+\lambda_1(K+L)\le \lambda_1(K)+\lambda_1(L).
\end{equation}
for centrally symmetric convex bodies $K,L\subset\R^n$.  The proof is an immediate application of \eqref{eq:MNT-intro} to Brownian path space. 

The same argument applies to any centered Gaussian process.
\begin{theorem}\label{thm:OU-schmuck}
	Let $K,L\subset\R^n$ be centrally symmetric convex bodies.  Then the first Ornstein--Uhlenbeck eigenvalues satisfy
	\[
	\lambda^{\OU}_1(K\cap L)+\lambda^{\OU}_1(K+L)
	\le
	\lambda^{\OU}_1(K)+\lambda^{\OU}_1(L).
	\]
\end{theorem}
\subsection{The $p$-Laplacian}

Again we first consider the inradius $r(K)$, corresponding to $p\to \infty$. In this case it holds for $K, L$ centrally symmetric that
\[r(K+L)\geq r(K)+r(L), \quad r(K\cap L)=\min(r(K), r(L))\]
and consequently 
\begin{equation}\label{eq:inradius1} r^{-1}(K\cap L)+r^{-1}(K+L)\leq r^{-1}(K)+r^{-1}(L),\end{equation}
but in fact one has

\begin{equation}\label{eq:inradius2}
	\lim_{p\to\infty} \left(\lambda_{1,p}(K)+\lambda_{1,p}(L)\right)^{1/p}= \max(r^{-1}(K), r^{-1}(L))=r^{-1}(K\cap L).
\end{equation}

Eq. \eqref{eq:inradius2} together with the improved Schmuckenschl\"ager inequality suggest the following.

\begin{question}\label{que:p_schmuckenshlager}
	Does the inequality
	\[
	\lambda_{1,p}(K\cap L)
	\le
	\lambda_{1,p}(K)+\lambda_{1,p}(L)
	\]
	hold for centrally symmetric convex bodies and $1< p<\infty$, $p\neq 2$?
	
	Does the stronger inequality \[
	\lambda_{1,p}(K\cap L)+\lambda_{1,p}(K+L)
	\le
	\lambda_{1,p}(K)+\lambda_{1,p}(L)
	\] hold?
\end{question} 

We will show a weaker inequality which nevertheless implies the first inequality in Question \ref{que:p_schmuckenshlager} in the limit $p\to 1$.

\begin{theorem}\label{thm:p-schmuck}
	Let $1< p<\infty$, and let $K,L\subset\R^n$ be centrally symmetric convex bodies.  Then
	\begin{equation}\label{eq:p-schmuck}
		\lambda_{1,p}(K\cap L)^{1/p}
		\le
		\lambda_{1,p}(K)^{1/p}+\lambda_{1,p}(L)^{1/p},
	\end{equation}
	and 
	\[      h(K\cap L)\leq
	h(K)+h(L).\]
\end{theorem}

\subsection{Proofs}
\subsubsection{The Dirichlet Laplacian}

\begin{proof}[Proof of Theorem \ref{thm:improved-schmuck}]
	Let $X_t$ be Brownian motion started at $0$.  For fixed $T$ and $m$, let $\widetilde\gamma_m$ be the law of $X^m=(X_{T/m},\ldots,X_T)$.  Apply \eqref{eq:MNT-intro} in the Gaussian space $(\R^n)^m$ to the centrally symmetric convex sets $K^m$ and $L^m$.  Since
	\[
	K^m\cap L^m=(K\cap L)^m,
	\qquad
	K^m+L^m=(K+L)^m,
	\]
	we get
	\[
	\widetilde\gamma_m(K^m)\widetilde\gamma_m(L^m)
	\le
	\widetilde\gamma_m((K\cap L)^m)
	\widetilde\gamma_m((K+L)^m).
	\]
	Taking logarithms, passing $m\to\infty$, dividing by $-T$, and then letting $T\to\infty$ gives
	\[
	\lambda_1(K)+\lambda_1(L)
	\ge
	\lambda_1(K\cap L)+\lambda_1(K+L).
	\]
\end{proof}

\subsubsection{The Ornstein--Uhlenbeck operator}

\begin{proof}[Proof of Theorem \ref{thm:OU-schmuck}]
	Let $Y_t$ be the Ornstein--Uhlenbeck process starting at $0$.  Its finite-dimensional joint distributions are centered Gaussian. For a convex body $K$, we have
	\[
	\lambda_1^{\OU}(K)=
	-\lim_{T\to\infty}\frac1T\log\mathbb P_0(\tau_K>T).
	\]

	Discretize the event $\{\tau_K>T\}$ at times $jT/m$.  Since the vector $(Y_{T/m},\ldots,Y_T)$ has a centered Gaussian distribution, the Gaussian inequality \eqref{eq:MNT-intro} applied to $K^m$ and $L^m$ gives
	\[
	\mathbb P_0(Y_{jT/m}\in K\ \forall j)
	\mathbb P_0(Y_{jT/m}\in L\ \forall j)
	\le
	\mathbb P_0(Y_{jT/m}\in K\cap L\ \forall j)
	\mathbb P_0(Y_{jT/m}\in K+L\ \forall j).
	\]
	Passing to continuous survival and then to the large-time limit gives the theorem.
\end{proof}

\subsubsection{A weak Schmuckenschl\"ager inequality for the $p$-Laplacian}
We follow the proof of Schmuckenschl\"ager, replacing bilinearity with the Minkowski inequality.

\begin{proof}[Proof of Theorem \ref{thm:p-schmuck}]
	Let $u$ and $v$ be positive first $p$-eigenfunctions on $K$ and $L$, respectively, extended by zero outside their domains and normalized by
	\[
	\int_Ku^p=\int_Lv^p=1.
	\]
	For $z\in\R^n$ set
	\[
	h_z(x)=u(x)v(x+z).
	\]
	Then $h_z$ is supported in $K\cap(L-z)$ and
	\[
	\int_{\R^n}\int_{\R^n} h_z(x)^p\,dx\,dz=1.
	\]
	Moreover,
	\[
	\nabla h_z(x)=v(x+z)\nabla u(x)+u(x)\nabla v(x+z).
	\]
	By Minkowski's inequality,
	\[
	\begin{aligned}
		\left(\int\!\int |\nabla h_z(x)|^pdx dz\right)^{1/p}
		&\le
		\left(\int\!\int v(x+z)^p|\nabla u(x)|^pdx dz\right)^{1/p} \\
		&\quad+
		\left(\int\!\int u(x)^p|\nabla v(x+z)|^pdx dz\right)^{1/p} \\
		&=\lambda_{1,p}(K)^{1/p}+\lambda_{1,p}(L)^{1/p}.
	\end{aligned}
	\]
	
	Thus
	
	\[ \int_{\R^n} \left(  \int_{\R^n} |\nabla h_z(x)|^pdx   - (\lambda_{1,p}(K)^{1/p}+\lambda_{1,p}(L)^{1/p})^p\int _{\R^n}h_z(x)^pdx \right)dz\leq 0.\]
	Hence there exists $z$ such that
	\[
	\lambda_{1,p}(K\cap(L-z))^{1/p}
	\le
	\lambda_{1,p}(K)^{1/p}+\lambda_{1,p}(L)^{1/p}.
	\]
	Because $K$ and $L$ are centrally symmetric,
	\[
	\lambda_{1,p}(K\cap(L+z))=
	\lambda_{1,p}(K\cap(L-z)).
	\]
	Also,
	\[
	\frac12(K\cap(L-z))+\frac12(K\cap(L+z))\subset K\cap L.
	\]
	The Brunn--Minkowski inequality for the first $p$-Laplacian eigenvalue \cite{ColesantiCuoghiSalani} states that $\lambda_{1,p}^{-1/p}$ is concave under Minkowski addition of convex bodies.  Therefore
	\[
	\begin{aligned}
		\lambda_{1,p}(K\cap L)^{-1/p}
		&\ge
		\frac12\lambda_{1,p}(K\cap(L-z))^{-1/p}
		+\frac12\lambda_{1,p}(K\cap(L+z))^{-1/p} \\
		&\ge
		\bigl(\lambda_{1,p}(K)^{1/p}+\lambda_{1,p}(L)^{1/p}\bigr)^{-1}.
	\end{aligned}
	\]
	This proves \eqref{eq:p-schmuck} for $1<p<\infty$. For the Cheeger constant, we let $p\to 1$ and use \cite{one_laplacian}.

\end{proof}

\section{Discussion and open questions}

\subsection{Reverse Faber--Krahn inequality}

The normalized minimal first Dirichlet eigenvalue has the sharp lower bound
\begin{equation}\label{eq:lower-FK}
	\Lambda_{\FK}(K)\geq \Lambda_{\FK}(B^n)=
	\lambda_1(B^n)|B^n|^{2/n}
	\geq c n,
\end{equation}
by the Rayleigh-Faber--Krahn inequality, with $c>0$ a universal constant. 

As $\Lambda_{\FK}(K)$ is invariant under all invertible linear maps and translations and moreover continuous (see Proposition \ref{prop:FK_continuity}), it is natural to look for a maximizer. This question is mentioned in Bucur and Fragala \cite{BucurFragala2016}, who also solved the analogous question for the planar  Cheeger constant \cite{BucurFragala}.
In analogy with the reverse isoperimetric inequality of K. Ball \cite{Ball_reverse}, they conjectured that the sharp maximizer is the cube in the centrally symmetric class. It is equally natural to expect the simplex to maximize $\Lambda_{\FK}$ in the class of all convex bodies.

The following conjecture appeared in \cite{Schmuckenschlager}. It suggests all convex bodies in Faber--Krahn position have roughly the same normalized first Dirichlet eigenvalue.
\begin{conjecture}\label{conj:reverse-FK}
	There is a universal constant $C$ such that for every convex body $K\subset\R^n$,
	\[
	\Lambda_{\FK}(K)\le Cn.
	\]
\end{conjecture}

A near-optimal upper bound due to Klartag and E. Milman appearing in \cite{Bizeul_Klartag} asserts that $\Lambda_{\FK}(K)\le C n\log^2 n$ for all convex bodies $K$. By incorporating into their proof the optimal mean-width estimate obtained recently in \cite{Bizeul_M*} and \cite{Paouris_Pathak}, one logarithmic factor can be removed.

\begin{theorem}\label{thm:KM}
	There is a universal constant $C$ such that for every convex body $K\subset\R^n$,
	\[
	\Lambda_{\FK}(K)\le C n\log n.
	\]
	Moreover, if $K$ is in isotropic position then $\lambda_1(K)|K|^{2/n}\leq  C n\log n$.
\end{theorem}

 \subsection{Beyond central symmetry} 
It is natural to ask whether central symmetry assumptions can be replaced by a more general centering assumption.

The B-theorem requires central symmetry. To prove Theorem \ref{thm:main-logconvex-intro}, we circumvent this by proving a weaker concavity property of the Gaussian measure in a special case.
However,  we do not know a convex counterexample to the B-conjecture with its linear barycenter at the origin.

\begin{question}\label{q:centered-B}
	Let $K\subset\R^n$ be a convex body whose Lebesgue barycenter is at the origin.  Is  $t\longmapsto \gamma_n(e^tK)$
	log-concave?
	More generally, is $t\mapsto \gamma_n(e^{tS}K)$ log-concave for every symmetric matrix $S$?
\end{question}

		As for Schmuckenschl\"ager's inequality, while the correlation inequality and its improved version \cite{Milman_Nakamura_Tsuji} apply to general convex bodies with $0$ in their Gaussian centroid \cite{Nakamura_Tsuji, Milman_Nakamura_Tsuji}, this property does not persist when taking samples of a Brownian motion as in section \ref{sec:symmetric_log_convex}, corresponding to Cartesian powers of the body in a Gaussian space with non-diagonal covariance matrix.
		Furthermore, numerical simulations suggest that Schmuckenschl\"ager's inequality fails under the relaxed assumption that $K, L$ have a common barycenter. One can nevertheless ask 
		
	\begin{question}
	Let $K, L\subset\R^n$ be convex bodies with Lebesgue barycenters at the origin.  Find the optimal constants $C_n, C_n'$ such that
	\[\lambda_1(K\cap L)\leq C_n(\lambda_1(K)+\lambda_1(L))\]
	or for the improved inequality
	\[\lambda_1(K\cap L)+\lambda_1(K+L)\leq C_n'(\lambda_1(K)+\lambda_1(L)).\]
\end{question}
As $-K\subset nK$, $-L\subset nL$, one immediately deduces that $C_n\leq n^2$:
\begin{align*}\lambda_1(K\cap L)&\leq \lambda_1(K\cap (-K)\cap L \cap (-L))\leq \lambda_1(K\cap (-K))+\lambda_1(L\cap (-L))\\&\leq \lambda_1(\frac1nK)+\lambda_1(\frac1nL) =n^2(\lambda_1(K)+\lambda_1(L)).\end{align*}
Using the Brunn--Minkowski inequality $\lambda_1(K+L)\leq \frac18 (\lambda_1(K)+\lambda_1(L))$ we see that $C_n'\leq n^2+\frac18$.

\subsection{The first nonzero Neumann eigenvalue}
The Szeg\H{o}--Weinberger inequality asserts that the first nonzero Neumann eigenvalue $\mu_1(K)$ is maximized, among domains of fixed volume, by the Euclidean ball. 

Furthermore, $\mu_1(K)$ is comparable, up to dimension-dependent constants, to $\mathrm{Diam}(K)^{-2}$. Given $P\neq I_n$ positive definite in $\SL(n)$, $t\mapsto \mathrm{Diam}(P^tK)$ is log-convex and proper, and in particular $t\mapsto \mu_1(P^tK)^{-1}$ is a proper function.

In light of the above and the log-convexity of $\lambda_1(P^tK)$,  it is then natural to ask

\begin{question}
	Is it true that $t\mapsto \mu_1(P^tK)$ is \emph{log-concave}, whenever $K$ is a convex body and $P$ positive definite? 
\end{question}
\noindent The example of the square shows that it need not be strictly log-concave. 

A closely related question is the existence of a distinguished Szeg\H{o}-Weinberger position of convex bodies. The existence of $A\in \SL(n)$ maximizing $\mu_1(AK)$ follows immediately from the properness of $t\mapsto \mu_1(P^tK)^{-1}$, as in Proposition \ref{prop:minimum_existence}.
Uniqueness is less clear.
\begin{question}
	Is the maximum of $\mu_1(AK)$, $A\in\SL(n)$, attained at a unique $A$ up to an orthogonal transformation?
\end{question}

\bibliographystyle{plain}
\bibliography{faber_krahn_references}

\end{document}